\documentclass[11pt,reqno]{amsart}
\usepackage[margin=1in]{geometry}
\usepackage{amsmath,amssymb,amsthm,graphicx,amsxtra, setspace}
\usepackage[utf8]{inputenc}
\usepackage{mathrsfs}
\usepackage{hyperref}
\usepackage{upgreek}
\usepackage{mathtools}
\usepackage{xcolor}
\usepackage{stmaryrd}

\usepackage[mathcal]{euscript}
\allowdisplaybreaks

\DeclareMathAlphabet{\mathpzc}{OT1}{pzc}{m}{it}

\usepackage[cyr]{aeguill}

\colorlet{darkblue}{blue!50!black}

\hypersetup{
	colorlinks,%
	citecolor=blue,%
	filecolor=red,%
	linkcolor=red,%
	urlcolor=blue,%
	pdfnewwindow=true,%
	pdfstartview={FitH}
}

\newtheorem{theorem}{Theorem}[section]
\newtheorem{lemma}[theorem]{Lemma}
\newtheorem{proposition}[theorem]{Proposition}

\newtheorem{definition}[theorem]{Definition}

\newtheorem{remark}[theorem]{Remark}

\newtheorem{hypothesis}[theorem]{Hypothesis}

\allowdisplaybreaks

\let\originalleft\left
\let\originalright\right
\renewcommand{\left}{\mathopen{}\mathclose\bgroup\originalleft}
\renewcommand{\right}{\aftergroup\egroup\originalright}

\renewcommand{\d}{\/\mathrm{d}\/}

\def\p{\Phi}

\def\w{\textbf{W}^{\varepsilon}_{{\theta}^{\varepsilon}}}

\def\x{\boldsymbol{x}}
\def\x{\boldsymbol{x}}

\def\p{\boldsymbol{p}}

\def\z{\boldsymbol{z}}
\def\v{\boldsymbol{v}}
\def\w{\boldsymbol{w}}
\def\A{\mathcal{A}}

\def\C{\mathrm{C}}
\def\B{\mathcal{B}}

\def\0{\boldsymbol{0}}
\def\y{\boldsymbol{y}}

\def\x{\boldsymbol{x}}
\def\bz{\boldsymbol{z}}

\def\w{\boldsymbol{w}}
\def\W{\boldsymbol{W}}

\def\H{\mathrm{\mathbb{H}}}

\def\f{\boldsymbol{f}}
\def\V{\mathbb{V}}
\def\U{\boldsymbol{U}}
\def\bu{\boldsymbol{u}}
\def\bv{\boldsymbol{v}}

\newcommand{\R}{\mathbb{R}}

\renewcommand{\d}{\/\mathrm{d}\/}

\makeatletter

\def\@setauthors{%
	\begingroup
	\trivlist
	\centering
	\footnotesize
	\@topsep30\p@
	\advance\@topsep by -\baselineskip
	\item\relax
	\author@andify\authors
	\def\\{\protect\linebreak}%
	{\normalfont\authors}%
	\endtrivlist
	\endgroup
}

\AtEndDocument{}

\makeatother

\begin{document}
	
\title[Optimal Control of 2D NSE with constraints]
{Optimal Control of 2D Navier--Stokes Equations with Constraints of Finite Codimension}

\author[M. T. Mohan]{%
{\normalsize	Manil T. Mohan}\textsuperscript{1*}\\[4pt]
	\textnormal{\small\textsuperscript{1}Department of Mathematics, Indian Institute of Technology Roorkee}\\
	\textnormal{\small Roorkee--247667, India}\\
	\textnormal{\small *Corresponding author}\\
	\textnormal{\small \textit{E-mail:} maniltmohan@ma.iitr.ac.in, maniltmohan@gmail.com}
}

\maketitle
	
%
%
%
%
%

	\begin{abstract}
	We study an optimal control problem for the two-dimensional Navier-Stokes equations projected onto the tangent space of the $\mathbb{H}$-unit sphere manifold $\mathcal{M} := \{\boldsymbol{u} \in \H : \|\boldsymbol{u}\|_{\H}=1\}$. This geometric projection yields a constrained evolution system of finite codimension, ensuring that the state remains on $\mathcal{M}$ for all times. Using a Faedo-Galerkin approximation technique, we establish the global existence and uniqueness of strong solutions to the constrained Navier-Stokes equations with external controls $\boldsymbol{U} \in \mathrm{L}^2(0,T;\H)$, together with uniform energy estimates. The associated optimal control problem, minimizing a quadratic cost functional subject to the manifold constraint, is shown to admit at least one optimal pair $(\boldsymbol{u}^*,\boldsymbol{U}^*)$. We derive the first-order necessary optimality conditions by formulating an
	optimal control system whose state-control pair $(\boldsymbol{u}^*,\boldsymbol{U}^*)$ evolves on the
	tangent bundle $T\mathcal{M}$ of the constraint manifold
	$\mathcal{M}$. In this setting, the control satisfies
	$\boldsymbol{U}^*(t)\in T_{\boldsymbol{u}^*(t)}\mathcal{M}$ for all $t\in[0,T]$, ensuring that the
	dynamics remain confined to the tangent bundle $T\mathcal{M}$. Furthermore, we establish the Fr\'echet differentiability of the control-to-state mapping $\boldsymbol{U}\mapsto \boldsymbol{u}_{\boldsymbol{U}}$ in $\mathrm{C}([0,T];\H)\cap\mathrm{L}^2(0,T;\V)$, with the corresponding linearized state characterized as the unique weak solution of a linearized constrained Navier-Stokes system preserving the tangency condition. These results provide a rigorous analytical foundation for the optimal control of fluid flows constrained to nonlinear manifolds.
	\end{abstract}
	
	

	\section{Introduction}\label{sec1}\setcounter{equation}{0}
	
	Geometric partial differential equations (PDEs) form a rich interface between
	differential geometry and analysis, describing the evolution of functions,
	curves, or surfaces under laws determined by intrinsic geometric quantities.
	They provide a unified framework for understanding flows and fields defined on
	manifolds, where the dynamics reflect both analytic and geometric structures.
	A natural bridge arises between this theory and constrained Navier-Stokes
	equations (NSE), in which the fluid motion follows classical Navier-Stokes dynamics
	but is additionally restricted by geometric or physical constraints, for
	example, conservation of kinetic energy, evolution confined to a manifold, or
	the preservation of certain invariants. Such constraints endow the fluid system
	with an inherent geometric character, bringing it closer to the realm of
	geometric flows. This correspondence reveals a unifying perspective: both
	geometric PDEs and constrained fluid equations describe nonlinear evolutions
	governed not only by analytic and physical principles but also by the geometry
	of the underlying configuration space, thereby calling for advanced tools from
	geometric analysis and variational methods (\cite{DH-81}).

	\subsection{Literature survey} 
	As discussed in \cite{ZB+GD+MM-18}, the present work is motivated in part by
	two major research directions concerning evolution equations subject to
	constraints. The first line of inspiration originates from the study of
	Caglioti, et. al. \cite{EC+MP+FR-09}, who analyzed a modified form of
	the two-dimensional NSE in vorticity formulation.
	Their model incorporates two integral constraints, the conservation of kinetic
	energy $E(\omega)$ and the moment of inertia $I(\omega)$, leading to the
	evolution equation
	$$ \frac{\partial \omega}{\partial t} + \boldsymbol{u} \cdot \nabla \omega = \nu \, \mathrm{div}\left[ \omega \nabla\left( \log \omega - b \psi - a \frac{|x|^2}{2} \right) \right], $$
	where $\omega = \mathrm{curl}(\boldsymbol{u})$ denotes the vorticity,
	$\boldsymbol{u}$ is the velocity field, $\psi = (-\Delta)^{-1}\omega$ is the
	stream function, and $a(\omega)$, $b(\omega)$ are Lagrange multipliers enforcing
	the constraints. They established global existence and uniqueness of classical
	solutions for a suitable class of initial data and demonstrated that, as
	$t \to \infty$, the dynamics converge toward the unique stationary solution of
	a corresponding constrained variational problem.

	The second source of motivation stems from studies on heat flows subject to
	nonlinear constraints. In particular, Rybka \cite{PR-06} investigated the
	evolution of the heat equation on a nonlinear manifold characterized by a finite
	number of integral constraints on powers of the solution,
	\begin{align}\label{eqn-1}
		\mathcal{M}
		= \left\{
		u \in L^2(\Omega) \cap C(\Omega) :
		\int_{\Omega} u^k(x)\,dx = C_k, \ k = 1, \dots, N
		\right\},
	\end{align}
	where $\Omega \subset \mathbb{R}^2$ is a smooth bounded domain. He
	considered a projected version of the heat equation whose dynamics are confined
	to the tangent space of $\mathcal{M}$, proving the global existence and
	uniqueness of solutions together with their convergence to equilibrium as
	$t \to \infty$. In a related direction, Caffarelli and Lin \cite{LC+FL-09}
	developed a general framework for heat flows with geometric constraints. They
	first established global well-posedness and energy conservation for such
	constrained parabolic systems and subsequently extended their results to
	singularly perturbed nonlocal flows. Their principal finding showed that these
	perturbed solutions converge strongly to weak solutions of a limiting nonlocal
	heat equation taking values in a singular target manifold.
	
	Building on these developments, Ma and Cheng
	\cite{LM+LC-09, LM+LC-13} investigated nonlocal heat flows that preserve the
	$L^2$-norm on compact Riemannian manifolds. Their analysis encompassed the
	global existence, stability, and long-time behavior of positive solutions, as
	well as the derivation of sharp gradient estimates. Moreover, for porous
	medium-type equations, they established convergence of solutions toward
	eigenfunctions of the Laplace-Beltrami operator under suitable geometric
	conditions, employing analytical techniques such as Sobolev embeddings and
	Moser iteration.
	
	More recent progress has been achieved by Brze\'zniak and Hussain
	\cite{ZB+JH-24}, who analyzed a nonlinear heat equation with a polynomial
	nonlinearity constrained to the unit sphere in a Hilbert space. By projecting
	the dynamics onto the tangent space and employing semigroup theory together
	with fixed-point techniques, they proved the existence of global strong
	solutions that remain confined to the manifold. Subsequently, Hussain
	\cite{JH-23} extended this framework to general Hilbert manifolds, establishing
	well-posedness of constrained heat equations through Faedo-Galerkin
	approximations and compactness methods. Both studies required restrictions on
	the growth of the nonlinearity, dictated by Sobolev embedding conditions, a
	limitation that the recent work \cite{AB+ZB+MTM-2025} aims to overcome

Further important progress was made by Antonelli et al. \cite{PA+PC+BS-24}, who established the well-posedness of a norm-preserving nonlinear heat equation on both bounded and unbounded domains in dimensions up to four. In particular, they demonstrated convergence of positive solutions toward the ground state when the problem is posed on a ball. Subsequently, Shakarov \cite{BS-25} investigated related heat equations constrained to the $L^2$-sphere, proving global well-posedness in bounded domains as well as in the whole space by means of Schauder fixed-point and contraction mapping techniques, albeit under restrictions on the nonlinearity exponent. More recently, Bawalia et al. \cite{AB+ZB+MTM-2025} proved global existence and uniqueness of strong solutions taking values in $L^p \cap H_0^1$ for $2 \leq p < \infty$ to a nonlinear heat equation with constraints on bounded domains in any spatial dimension $d \geq 1$. They also established the invariance of the unit $L^2$-sphere manifold and demonstrated convergence of positive solutions to the ground state in the presence of damping nonlinearities. 

The paper \cite{BW+CY-2013} reviewed recent advances in the mathematical theory and numerical methods for Bose-Einstein condensation (BEC), formulated within the framework of the Gross-Pitaevskii equation (GPE).
Brze\'zniak and Cerrai \cite{ZB+SC-25} studied stochastic damped wave equations whose dynamics are restricted to the unit sphere of a Hilbert space, establishing both the well-posedness of the system and its asymptotic properties as the mass parameter tends to zero. More recently, Cerrai and Xie \cite{SC+MX-25} analyzed the Smoluchowski-Kramers approximation for such constrained stochastic wave equations, proving that in the vanishing-mass limit, the solutions converge to those of the corresponding constrained stochastic parabolic equation.

	Inspired by the seminal works \cite{LC+FL-09, EC+MP+FR-09, PR-06}, Brze\'zniak et al. \cite{ZB+GD+MM-18} investigated the NSE with constrained forcing on the domains $\mathbb{T}^2$ and $\mathbb{R}^2$. They established global existence and uniqueness of energy-preserving solutions and demonstrated convergence to solutions of the Euler equations in the vanishing viscosity limit. In a related development, \cite{PC+HF-25} analyzed the 2D NSE  on the torus $\mathbb{T}^2$ incorporating a bilinear control term that enforces evolution within prescribed hyperplanes through a nonlinear, nonlocal perturbation, proving the system’s well-posedness. The study of constrained systems has also been extended to the stochastic framework. In this context, Brze\'zniak and Hussain \cite{ZB+JH-20} initiated the analysis of stochastic nonlinear heat equations with Stratonovich noise under an $L^2$-norm constraint, establishing the existence and uniqueness of mild solutions in two dimensions. Subsequently, Brze\'zniak and Dhariwal \cite{ZB+GD-21} extended these ideas to the stochastic NSE with multiplicative Gaussian noise, proving the existence of martingale solutions and pathwise uniqueness via Yamada-Watanabe-type arguments.

		The regulation of fluid motion and turbulence within a prescribed domain, starting from given initial conditions and influenced through body forces, boundary controls, or temperature variations, has long been a fundamental problem in fluid mechanics (see, e.g., \cite{AF+TR-90,FAV-2000,MDG-2003,LJL-1971,SSS-1998,TF-2010}). With major advances in the theory of optimal control for infinite-dimensional nonlinear systems and partial differential equations, the study of optimal control problems governed by the NSE has expanded considerably over the past few decades. Numerous variants of these problems have been analyzed in the literature, including those presented in \cite{BT+TR+ZM-2000,CE+KK-2019,DJC+GR-2008,DM+IK-1994,FHO+SSS-1992,FHO+SSS-1994,GMD+MS-1999,HM+KK-2001,TF+WD-2006,WG-2002,WG+WL-2003} and the references therein.

		Barbu and Sritharan \cite{Barbu+Sritharan-2001} (see \cite{Gautam+Kinra+Mohan-2024}) investigated the flow-invariance problem for the controlled Navier-Stokes equations, developing feedback controllers designed to preserve prescribed physical constraint sets, including level sets associated with invariants of the corresponding inviscid flow. Their analysis is based on \(m\)-accretive operator theory, nonsmooth analysis, and nonlinear semigroup methods, and the resulting controlled dynamics are formulated as a Navier-Stokes inclusion involving a multivalued feedback term. In contrast, the present work addresses an optimal control problem for the two-dimensional Navier–Stokes equations constrained to the nonlinear \(L^2\)-sphere. We formulate the constrained dynamics through the orthogonal projection onto the tangent space of the sphere, establish global existence and uniqueness of strong solutions, prove the existence of an optimal control, and derive first-order necessary optimality conditions. We further establish the Fréchet differentiability of the control-to-state mapping and analyze the associated linearized constrained system. Thus, while both works involve constraint preservation in controlled Navier-Stokes dynamics, the objectives and analytical frameworks are fundamentally different.
		
		\subsection{Aims, difficulties, approaches and novelties} 
		The primary aim of this work is to develop a rigorous analytical and geometric framework for the optimal control of the two-dimensional NSE constrained to evolve on a finite-codimensional manifold in the underlying Hilbert space $\H$ (see Subsection \ref{sec2.1}). Specifically, the state variable $\bu(t),\ t\in[0,T]$ (velocity field) is constrained to lie on the manifold
		$$
		\mathcal{M} = \{\bu \in \H : \|\bu\|_{\H} = 1\},
		$$
		which represents the preservation of kinetic energy throughout the evolution. By projecting the classical Navier-Stokes dynamics onto the tangent space $T_{\bu}\mathcal{M}$, we obtain a constrained system (see \eqref{eqn-constrained-NSE-N}) that naturally evolves on the \emph{tangent bundle}  \(T\mathcal M=\bigsqcup_{\bu\in\mathcal M}T_{\bu}\mathcal M,\) which 
		is a fixed manifold (the disjoint union of the tangent spaces). The main objectives of this paper are to establish the existence and uniqueness of strong  solutions to the constrained Navier-Stokes system (Theorem \ref{thm-exis-unique}), to prove the existence of an optimal control (Theorem \ref{thm-exis-optimal}) minimizing a suitable quadratic cost functional (see \eqref{cost}), and to derive first-order necessary optimality conditions in terms of a coupled state-adjoint-control system that evolves consistently within the tangent bundle (Theorem \ref{main}).
		
		The study of such constrained systems introduces several analytical and geometric challenges. The nonlinear constraint $\|\bu(t)\|_{\H}=1$ defines a nonconvex manifold, and consequently, the feasible set of states is not a linear subspace, rendering standard convex-analytic techniques inapplicable. The geometry of the problem is further complicated by the fact that the projection operator $\pi_{\bu(t)}$ depends nonlinearly on the evolving state, coupling the dynamics and geometry in a highly nontrivial way. Ensuring that both the control and linearized perturbations remain tangent to $\mathcal{M}$ at all times (Theorem \ref{linear} and Proposition \ref{lem3.8}) requires delicate consistency arguments that go beyond classical approaches for unconstrained fluid control problems.
		
		To address these difficulties, we employ a combination of functional analytic, geometric, and variational techniques. The \emph{Faedo-Galerkin approximation}, together with compactness tools such as the \emph{Banach-Alaoglu} and \emph{Aubin-Lions-Simon} theorems, as well as the absolute continuity result provided by the \emph{Lions-Magenes} lemma, is employed to construct global strong solutions that remain on the constraint manifold $\mathcal{M}$ (Theorem \ref{thm-exis-unique}). This approach stands in contrast to the method used in \cite{ZB+GD+MM-18}, where the authors applied semigroup theory and fixed-point arguments to establish the corresponding results for the uncontrolled two-dimensional constrained NSE. The \emph{adjoint system} is derived using variational differentiation and projection arguments consistent with the manifold structure (Theorem \ref{thm3.4}).  The resulting optimal control is shown to be the \emph{projection of the adjoint variable} onto the tangent space,
		$$\U^*(t) = -\pi_{\bu^*(t)}(\p(t)) = -\p(t) + (\p(t),\bu^*(t))\bu^*(t),\ t\in[0,T],$$
		ensuring that the state-control pair $(\bu^*(t), \U^*(t))$ evolves entirely within the tangent bundle $T\mathcal{M}$ (Theorem \ref{main}). This provides a natural geometric interpretation of the constrained dynamics: the adjoint variable encodes sensitivity information in the ambient space, while the optimal control acts as its tangential projection that respects the manifold constraint. The novelty of this work lies in formulating and analyzing a Navier-Stokes optimal control problem on a nonlinear manifold, integrating geometric PDE methods with optimal control theory, and revealing how the intrinsic geometry of the constraint governs both the analytical structure and the optimality system of the problem.

	Although the first-order necessary optimality conditions were formally derived in \cite{SS-2023}, the author defined the admissible control set as
	$\mathscr{U}_{ad}:=\left\{\U\in T_{\bu}\mathcal{M}:\|\U\|_{\V}\ \text{ is bounded} \right\}.$
	However, the tangent space $T_{\bu}\mathcal{M}$ depends on the state $\bu(t)$ and therefore varies with time. As a result, $\mathscr{U}_{\mathrm{ad}}$ becomes \emph{state-dependent}, meaning it is not a fixed linear space. Consequently, the admissible control set changes as the state evolves, rendering it \emph{nonlinear} and \emph{nonconvex}. In contrast, standard optimal control theory for partial differential equations requires the admissible control set $\mathscr{U}_{\mathrm{ad}}$ to be a fixed, convex, and closed subset of the control space, independent of the state, to ensure the validity of existence and differentiability results. The choice $\U \in T_{\bu}\mathcal{M}$ therefore breaks these foundational assumptions: the cost functional $\mathscr{J}(\bu, \U)$ is no longer convex in $\U$, standard existence proofs for optimal controls fail, and the differentiability of the control-to-state mapping cannot be guaranteed. To overcome these issues, we instead define the control space as	$	\mathscr{U}_{\mathrm{ad}} = \mathrm{L}^2(0,T;\mathbb{H}),	$	which is fixed, convex, and independent of the state, and then project each $\U \in \mathscr{U}_{\mathrm{ad}}$ onto the tangent space $T_{\bu}\mathcal{M}$ within the dynamics. This approach preserves the energy constraint while maintaining the mathematical structure necessary for standard optimal control analysis, such as the existence of optimal controls, differentiability of the control-to-state mapping, and the derivation of adjoint equations. This formulation allows the application of the well-established analytical framework for existence, differentiability, and adjoint systems in optimal control theory. To the best of our knowledge, this is the first work to provide a systematic and mathematically consistent treatment of the optimal control problem for the constrained 2D NSE.

		\subsection{Organization of the paper} The remainder of the paper is organized as follows. In Section \ref{sec2}, we introduce the necessary functional framework that forms the basis for deriving the main results. In this section, we also establish the existence and uniqueness of strong solutions to the forced two-dimensional NSE subject to finite-codimensional constraints (Theorem \ref{thm-exis-unique}). Section \ref{sec3} is devoted to the existence of an optimal control that minimizes a suitable cost functional (Theorem \ref{thm-exis-optimal}). In the same section, we also prove the existence and uniqueness of weak solutions to the linearized and adjoint systems (Theorems \ref{linear} and \ref{thm3.4}). Section \ref{sec4} focuses on  the first-order necessary optimality conditions for the optimal control problem \ref{control problem} (Theorem \ref{main}). To establish Theorem \ref{main}, we further demonstrate the Fr\'echet differentiability of the control-to-state mapping in the same section (Lemma \ref{lem3.7} and Proposition \ref{lem3.8}).

		\section{Existence and Uniqueness}\label{sec2}\setcounter{equation}{0}
	The purpose of this section is to establish the existence and uniqueness of strong solutions to the forced two-dimensional NSE subject to finite-codimensional constraints. To this end, we first introduce the necessary functional framework that will serve as the foundation for deriving the main results of the paper.
	
	\subsection{Functional setting}\label{sec2.1}
Let $\mathbb{T}^2 $ denote the two-dimensional torus, which can be viewed as a bounded periodic domain. For $p \in [1, \infty] $  and $ k \in \mathbb{N} $, we define the vector-valued Lebesgue and Sobolev spaces on $ \mathbb{T}^2 $ as follows:
$
\mathbb{L}^p(\mathbb{T}^2) := \mathrm{L}^p(\mathbb{T}^2; \mathbb{R}^2), \ 
\mathbb{W}^{k,p}(\mathbb{T}^2) := \mathrm{W}^{k,p}(\mathbb{T}^2; \mathbb{R}^2).
$
The standard inner product on $\mathbb{L}^2(\mathbb{T}^2) $ is denoted by
$(\bu, \bv) = \int_{\mathbb{T}^2} \bu(x) \cdot \bv(x), dx, \  \bu, \bv \in \mathbb{L}^2(\mathbb{T}^2),$
and the associated norm is written as $ \|\bu\|_{\mathbb{L}^2} .$
We further set
$\mathbb{H}^k := \mathbb{W}^{k,2}(\mathbb{T}^2; \mathbb{R}^2),$
and denote its norm by $\|\cdot\|_{\mathbb{H}^k} $.
In particular, the scalar product on $ \mathbb{H}^1(\mathbb{T}^2) $ is given by
$(\bu, \bv)_{\mathbb{H}^1} = (\bu, \bv) + (\nabla \bu, \nabla \bv), \  \bu, \bv \in \mathbb{H}^1(\mathbb{T}^2),$
which induces the norm
$\|\bu\|_{\mathbb{H}^1} = \left(\|\bu\|_{\mathbb{L}^2}^2 + \|\nabla \bu|_{\mathbb{L}^2}^2 \right)^{1/2}, \  \bu \in \mathbb{H}^1(\mathbb{T}^2).$
We now introduce the following function spaces:

	\begin{align*}
		\mathbb{L}^2_0:=\left\{\bu\in\mathbb{L}^2(\mathbb{T}^2):\int_{\mathbb{T}^2}\bu(x)\d x=\boldsymbol{0}\right\},\ 
		\mathbb{H}:=\left\{\bu\in\mathbb{L}^2_0:\mathrm{div\ }\bu=0\right\},\ 
		\mathbb{V}:=\mathbb{H}^1\cap\mathbb{H}. 
	\end{align*}
	We equip the space $\mathbb{H}$  with the standard $ \mathbb{L}^2 $-inner product
$	(\bu, \bv) = \int_{\mathbb{T}^2} \bu(x) \cdot \bv(x), dx,$
	and the corresponding norm $\|\bu\|_{\mathbb{H}} = \|\bu\|_{\mathbb{L}^2}$, for all $ \bu, \bv \in \mathbb{H} $. On the space $\mathbb{V} $, we define the inner product
$	(\!(\bu, \bv)\!) := (\nabla \bu, \nabla \bv),$
	and the associated norm
$	\|\bu\|_{\mathbb{V}} := \|\nabla \bu\|_{\mathbb{H}}, \  \bu, \bv \in \mathbb{V}.$
	Furthermore, by the Poincar\'e-Wirtinger inequality on the two-dimensional torus,
$	\|\bu\|_{\mathbb{H}} \leq \frac{1}{\sqrt{\lambda_1}} \|\nabla \bu\|_{\mathbb{H}}, \ \text{ for all }\ \bu \in \mathbb{V},$
	where $\lambda_1 > 0 $ denotes the smallest (nonzero) eigenvalue of the Stokes operator (see Section \ref{sec2.2} below).
	Consequently, the norm $ \|\cdot\|_{\mathbb{V}} $ is equivalent  to the standard $\mathbb{H}^1$-norm on $\mathbb{V}$.

	\subsection{Linear operator}\label{sec2.2}
	Let $\mathrm{P}_{\H} : \mathbb{L}^2(\mathbb{T}^2) \to \H$ denote the \emph{Leray-Helmholtz projection operator} (see, e.g., \cite[Chapter 2]{RJC+RJL+SW-2016}), which projects square-integrable vector fields onto the subspace of divergence-free ones.
	We define the \emph{Stokes operator} $\A : \mathrm{D}(\A) \to \H$ by
	\begin{align*} 
		\A\bu=-\mathrm{P}_{\H}\Delta\bu, \ \bu\in \mathrm{D}(\A)=\H\cap\mathbb{H}^2(\mathbb{T}^2). 
		\end{align*}
	It is well known that $\A$ is a self-adjoint, nonnegative operator on $\H$, and that it commutes with the projection $\mathrm{P}_{\H}$.
	Moreover, the domain of its square root coincides with the space $\V$, and for all $\bu \in \V$,
\begin{align*}\mbox{ $\mathrm{D}(\A^{\frac{1}{2}})=\V$ \ and\ $\langle\A\bu,\bu\rangle=\|\nabla\bu\|_{\H}^2=\|\A^{\frac{1}{2}}\bu\|_{\H}^2,$ \ for \ $\bu\in\V$. } \end{align*}
	Since the torus $\mathbb{T}^2$ is compact, the inverse $\A^{-1}$ is a compact operator on $\H$. Consequently, the spectrum of $\A$ consists of a discrete sequence of eigenvalues of finite multiplicity accumulating only at $+\infty$.
	Denoting these eigenvalues by
	$0<\lambda_1\leq\lambda_2\leq \ldots,\ \lambda_n\to+\infty$,
	we may choose the corresponding eigenfunctions $\{\w_n\}_{n=1}^{\infty}$ to form an orthonormal basis of $\H$. Every $\bu \in \H$ then admits the Fourier-type expansion
	$\bu=\sum_{k=1}^{\infty}(\bu,\w_k)\w_k$
	and the action of $\A$ is diagonal in this basis:
	$\A\bu = \sum_{k=1}^{\infty}\lambda_k(\bu,\w_k)\w_k$ for all $\bu\in\mathrm{D}(\A)$.
	Hence,
\begin{align*} \|\nabla\bu\|_{\H}^2= \langle\A\bu,\bu\rangle =\sum_{k=1}^{\infty}\lambda_k|(\bu,\w_k)|^2\geq\lambda_1 \sum_{k=1}^{\infty}|(\bu,\w_k)|^2=\lambda_1\|\bu\|_{\H}^2, \end{align*}
	which yields the \emph{Poincar\'e-Wirtinger inequality}.

		\subsection{Bilinear operator}
	We define the \emph{trilinear form} $b(\cdot,\cdot,\cdot):\V\times\V\times\V\to\R$ as follows:  $$b(\bu,\bv,\w)=\int_{\mathbb{T}^2}(\bu(x)\cdot\nabla)\bv(x)\cdot\w(x)\d x=\sum_{i,j=1}^2\int_{\mathbb{T}^2}\bu_i(x)\frac{\partial \bv_j(x)}{\partial x_i}\w_j(x)\d x.$$ If $\bu, \bv \in \V$ are such that the mapping $b(\bu, \bv, \cdot)$ defines a continuous linear functional on $\V$, we denote the corresponding element of the dual space $\V^{\prime}$ by $\mathcal{B}(\bu, \bv)$.
	In particular, we set $ \mathcal{B}(\bu) = \mathcal{B}(\bu, \bu) := \mathcal{P}\big[(\bu \cdot \nabla)\bu\big],$ for $\bu\in D(\mathcal{A}).$
	By performing integration by parts, one obtains the following relation:
	\begin{equation*}
		\left\{
		\begin{aligned}
			b(\bu,\bv,\w) &=  -b(\bu,\w,\bv),\ &&\text{ for all }\ \bu,\bv,\w\in \V,\\
			b(\bu,\bv,\bv) &= 0,\ &&\text{ for all }\ \bu,\bv \in\V.
		\end{aligned}
		\right.\end{equation*}
	By  applying H\"older's, Ladyzhenskaya's (\cite[Lemma 1, Chapter 1]{OAL-1969}) and Poincar\'e's inequalities, we estimate 
	\begin{align*}
		|b(\bu,\bv,\w)|=|b(\bu,\w,\bv)|&\leq \|\bu\|_{\mathbb{L}^4}\|\bv\|_{\mathbb{L}^4}\|\w\|_{\V}\leq\sqrt{2}\|\bu\|_{\H}^{1/2}\|\bu\|_{\V}^{1/2}\|\bv\|_{\H}^{1/2}\|\bv\|_{\V}^{1/2}\|\w\|_{\V}\nonumber\\&\leq\sqrt{\frac{2}{\lambda_1}}\|\bu\|_{\V}\|\bv\|_{\V}\|\w\|_{\V},
	\end{align*}
	for all $\bu,\bv,\w\in\V$ and $\w\in\V$. Therefore, we deduce 
	\begin{align}\label{2p9}
		\|\mathcal{B}(\bu,\bv)\|_{\V'}\leq \sqrt{\frac{2}{\lambda_1}}\|\bu\|_{\V}\|\bv\|_{\V}.
	\end{align}
	If $\bu\in\V$, $\bv\in\mathrm{D}(\A)$ and $\w\in\H$, then applications of H\"older's and Ladyzhenskaya's inequalities yield 
	\begin{align}
		|b(\bu,\bv,\w)|\leq \|\bu\|_{\mathbb{L}^4}\|\nabla\bv\|_{\mathbb{L}^4}\|\w\|_{\H}\leq C\|\bu\|_{\H}^{1/2}\|\bu\|_{\V}^{1/2}\|\nabla\bv\|_{\H}^{1/2}\|\A\bv\|_{\H}^{1/2}\|\w\|_{\H}. 
	\end{align}
	Consequently, the trilinear form $b : \V \times \V \times \V \to \R$  admits a unique extension to a bounded trilinear mapping
$	b : \V \times \mathrm{D}(\A) \times \H \to \R,$
	which we continue to denote by the same symbol.
	Correspondingly, the operator $\B$ can be uniquely extended to a bounded linear operator
	$\B : \V \times \mathrm{D}(\A) \to \H.$
	Furthermore, on the torus $\mathbb{T}^2$, for every $\bu \in \mathrm{D}(\A)$, the following result holds (see \cite[Lemma 3.1]{TR-1995}):
	\begin{align}
		(\B(\bu,\bu),\A\bu)=0,
	\end{align}
	which plays a crucial role in our analysis.

	\subsection{Abstract formulation}
The two-dimensional Navier-Stokes equations defined on the torus $\mathbb{T}^2$
can be written as follows:
	\begin{equation}\label{eqn-2DNSE-normal}
		\left\{
		\begin{aligned}
			&\frac{\partial \bu(x,t)}{\partial t}-\nu\Delta \bu(x,t)+(\bu(x,t)\cdot\nabla)\bu(x,t)+\nabla p(x,t)=\f(x,t),\ (x,t)\in\mathbb{T}^2\times(0,T),\\
			&\nabla\cdot \bu(x,t)=0,\ (x,t)\in\mathbb{T}^2\times[0,T],\\
			&\bu(x,0)=\bu_0(x),\ x\in\mathbb{T}^2,
		\end{aligned}
		\right.
	\end{equation}
	where $\bu:\mathbb{T}^2\times(0,\infty)\to\R^2$ represents the velocity field, $p:\mathbb{T}^2\times(0,\infty)\to\R$ denotes the pressure, $\U:\mathbb{T}^2\times(0,\infty)\to\R^2$ is an external forcing and $\nu>0$  is the viscosity of the fluid. Moreover, $\bu(\cdot,\cdot)$, $p(\cdot,\cdot)$ and $\f(\cdot,\cdot)$ satisfy the following periodic conditions:
	\begin{align}\label{2}
		\bu(x+\mathrm{L}e_{i},\cdot) = \bu(x,\cdot), \ p(x+\mathrm{L}e_{i},\cdot) = p(x,\cdot) \ \text{ and } \ \f(x+\mathrm{L}e_{i},\cdot) = \f(x,\cdot),
	\end{align}
	for every $x\in\R^{d}$ and $i=1,\ldots,d,$ where $\{e_{1},\dots,e_{d}\}$ is the canonical basis of $\R^{d}.$  The Navier-Stokes equations \eqref{eqn-2DNSE-normal} projected on divergence free vector field is given by
		\begin{equation}\label{eqn-2DNSE}
		\left\{
		\begin{aligned}
			\frac{\d\bu(t)}{\d t}&=-\nu\A\bu(t)-\B(\bu(t),\bu(t))+\U(t),\ t>0, \\
			\bu(0)&=\bu_0,
		\end{aligned}
		\right. 
	\end{equation} 
where $\bu_0\in\H$ and $\U:=\mathrm{P}_{\H}\f\in\mathrm{L}^2(0,T;\H)$. We introduce the set
\begin{align}
\mathcal{M}:=\left\{\bu\in\H:\|\bu\|_{\H}=1\right\},
\end{align}	
that is, the collection of all elements in $\H$ having unit norm. The tangent space to $\mathcal{M}$ at a point $\bu\in\mathcal{M}$ is given by $$T_{\bu}\mathcal{M}:=\left\{\bv\in\mathbb{H}:(\bv,\bu)=0\right\}.$$ In other words, $T_{\bu}\mathcal{M}$ consists of all vectors in $\H$ that are orthogonal to $\bu$. We introduce the linear operator $\pi_{\bu} : \H \to T_{\bu}\mathcal{M}$ defined by
$$
\pi_{\bu}(\bv) = \bv - (\bv, \bu) \bu.
$$
This operator represents the orthogonal projection of $\H$ onto the tangent space $T_{\bu}\mathcal{M}$.

For an initial condition $\bu_0 \in \mathcal{M}$, we seek solutions that remain on the manifold $\mathcal{M}$ for all time. To ensure this constraint, we substitute the term $-\nu \A \bu - \B(\bu, \bu) + \U$ with its orthogonal projection onto the tangent space $T_{\bu}\mathcal{M}$, that is, 
\begin{align}
	\pi_{\bu}(-\nu\A\bu-\B(\bu,\bu)+\U)&=-\nu\A\bu-\B(\bu,\bu)+\U-(-\nu\A\bu-\B(\bu,\bu)+\U,\bu)\bu\nonumber\\&=-\nu\A\bu+\nu\|\nabla\bu\|_{\H}^2\bu-\B(\bu,\bu)+\U-(\U,\bu)\bu,
\end{align}
since $\langle\B(\bu,\bu),\bu\rangle =0$ and $\bu\in\mathcal{M}$, $\B(\bu,\bu)\in T_{\bu}\mathcal{M}$. By projecting the Navier-Stokes equations \eqref{eqn-2DNSE} onto the tangent space $T_{\bu}\mathcal{M}$, we derive the two-dimensional Navier-Stokes system subject to finite-codimensional constraints, expressed as follows:
	\begin{equation}\label{eqn-constrained-NSE-N}
		\left\{
		\begin{aligned}
			\frac{\d\bu(t)}{\d t}&=-\nu\A\bu(t)+\nu\|\nabla\bu(t)\|_{\H}^2\bu(t)-\B(\bu(t),\bu(t))+\U(t)-(\U(t),\bu(t))\bu(t),\ t\in(0,T), \\
			\bu(0)&=\bu_0,
		\end{aligned}
		\right. 
	\end{equation} 
	where $\bu_0\in\V\cap\mathcal{M}$ and $\U\in\mathrm{L}^2(0,T;\H)$ is the control acting on the system.  The dynamics of the state $\bu(t)$ in \eqref{eqn-constrained-NSE-N} are modified by the term $-(\U(t),\bu(t))\bu(t),$ which acts as a \emph{geometric Lagrange multiplier} enforcing the motion to remain tangent to $\mathcal{M}$. The main advantage of formulating the problem \eqref{eqn-constrained-NSE-N} on the torus is that, in this setting, $(\B(\bu,\bu),\A\bu)=0$ for all $\bu\in \mathrm{D}(\A)$. For the case $\U=\boldsymbol{0},$ the authors in \cite{ZB+GD+MM-18} established the existence and uniqueness of global strong solutions to problem \eqref{eqn-constrained-NSE-N} using a fixed-point argument. Since the inclusion of the control term $\U$ introduces an additional nonlinearity, we establish the well-posedness of \eqref{eqn-constrained-NSE-N} by employing a Faedo-Galerkin approximation method and using compactness arguments.

	\begin{remark}
	If one considers the system \eqref{eqn-2DNSE}, then	the feedback control $\U=\nu\|\nabla\bu\|_{\H}^2\bu$ exactly compensates the viscous dissipation and enforces the conservation of the $\H$-norm. In fact, if $\|\bu_0\|_{\H}=1$, then $\|\bu(t)\|_{\H}=1$ for all $t\geq 0$. Thus the controlled dynamics evolves on the unit sphere in $\H$.
	\end{remark}
	
	\begin{lemma}[{\cite[Lemmas 3.2, 3.3]{ZB+GD+MM-18}}]\label{lem-loc-lip}
		The mapping 
		\begin{align}\label{eqn-F(u)}
		\mathrm{D}(\A)\ni\bu\mapsto\mathcal{F}(\bu):=-\nu\mathcal{A}\bu+ \nu\|\nabla\bu\|_{\H}^2\bu-\B(\bu,\bu)+\U-(\U,\bu)\bu\in \H
		\end{align}
		 is locally Lipschitz, that is, there exists a constant $L_{R,\|\U\|_{\H}}>0$ such that 
		\begin{align}\label{eqn-loc-lip}
			\|\mathcal{F}(\bu)-\mathcal{F}(\bv)\|_{\H}\leq L_{R,\|\U\|_{\H}}\|\bu-\bv\|_{\mathrm{D}(\A)},
		\end{align}
		for all $\|\bu\|_{\mathrm{D}(\A)}, \|\bv\|_{\mathrm{D}(\A)}\leq R$. 
	\end{lemma}
	\begin{proof}
	Let us fix $\bu,\bv\in \mathrm{D}(\A)$. 	Since $\mathcal{A}$ is a linear operator, we have 
		\begin{align}\label{eqn-11}
			\nu\|\A\bu-\A\bu\|_{\H}=	\nu\|\A(\bu-\bu)\|_{\H}=\nu\|\bu-\bv\|_{\mathrm{D}(\A)}. 
		\end{align}
		Let us now consider the nonlinear term 
		\begin{align}\label{eqn-12}
			\nu\left\|\|\nabla\bu\|_{\H}^2\bu-\|\nabla\bv\|_{\H}^2\bv\right\|_{\H}&\leq\nu \|\nabla\bu\|_{\H}^2\|\bu-\bv\|_{\H}+\left(\|\nabla\bu\|_{\H}^2-\|\nabla\bv\|_{\H}^2\right)\|\bv\|_{\H}\nonumber\\&\leq\frac{\nu}{\sqrt{\lambda_1}}\left(\|\nabla\bu\|_{\H}^2+\|\nabla\bu\|_{\H}\|\nabla\bv\|_{\H}+\|\nabla\bv\|_{\H}^2\right)\|\nabla(\bu-\bv)\|_{\H}. 
		\end{align}
		Using H\"older's, Agmon's  and Sobolev's inequalities, we estimate 
		\begin{align}\label{eqn-13}
			\|\B(\bu,\bu)-\B(\bv,\bv)\|_{\H}&\leq\|\B(\bu-\bv,\bu)\|_{\H}+\|\B(\bv,\bu-\bv)\|_{\H}\nonumber\\&\leq \|\bu-\bv\|_{\mathbb{L}^4}\|\nabla\bu\|_{\mathbb{L}^4}+\|\bv\|_{\mathbb{L}^{\infty}}\|\nabla(\bu-\bv)\|_{\H}\nonumber\\&\leq C\left(\|\nabla\bu\|_{\H}^{1/2}\|\A\bu\|_{\H}^{1/2}+\|\bv\|_{\H}^{1/2}\|\A\bv\|_{\H}^{1/2}\right)\|\nabla(\bu-\bv)\|_{\H}. 
		\end{align}
		Finally, we consider
		\begin{align}\label{eqn-14}
			\|\U-(\U,\bu)\bu-(\U-(\U,\bv)\bv)\|_{\H}&=\|(\U,\bu)\bu-(\U,\bv)\bv\|_{\H}
			\nonumber\\&\leq|(\U,\bu)|\|\bu-\bv\|_{\H}+|(\U,\bu-\bv)|\|\bv\|_{\H}
			\nonumber\\&\leq\|\U\|_{\H}\left(\|\bu\|_{\H}+\|\bv\|_{\H}\right)\|\bu-\bv\|_{\H}.
		\end{align}
		Combining \eqref{eqn-11}-\eqref{eqn-14} and using the fact that $\|\bu\|_{\mathrm{D}(\A)}, \|\bv\|_{\mathrm{D}(\A)}\leq R$, we obtain \eqref{eqn-loc-lip}. 
	\end{proof}
	
	\subsection{Existence and uniqueness}
	To establish the existence and uniqueness of strong solutions to problem \eqref{eqn-constrained-NSE-N}, we employ the Faedo-Galerkin approximation method together with the Banach-Alaoglu theorem, the Aubin-Lions-Simon compactness lemma, and the Lions-Magenes lemma. Before proceeding with the analysis, we first define what is meant by a strong solution of problem \eqref{eqn-constrained-NSE-N}.
	\begin{definition}\label{def-strong-soln}
		Let   $\bu_0 \in \V\cap\mathcal{M}$  be fixed.
		A function
		\begin{align*}
			\bu \in \C ([0,T];  \V\cap\mathcal{M})\cap \mathrm{L}^2(0,T; \mathrm{D}(\A)),
		\end{align*}
		is called a \emph{strong solution} of problem \eqref{eqn-constrained-NSE-N} on time interval $[0,T]$, if and only if
		the following three conditions are satisfied:
		\begin{itemize}
			\item[(i)] The time derivative of $\bu$, in the weak sense, $\frac{\d\bu}{\d t}$ is in $\mathrm{L}^2(0,T; \H)$ and $-\nu\A\bu+\nu\|\nabla\bu\|_{\H}^2\bu-\B(\bu,\bu)+\U-(\U,\bu)\bu \in \mathrm{L}^2(0,T; \H)$.
			\item[(ii)] The equation \eqref{eqn-constrained-NSE-N} is satisfied in $\mathrm{L}^2(0,T;\H)$, that is, for all $\boldsymbol{\varphi}\in\mathrm{L}^2(0,T;\H)$, $\bu$ satisfies 
		\begin{align}
				\int_0^T\left(\frac{\d\bu(t)}{\d t}-\mathcal{F}(\bu(t)),\boldsymbol{\varphi}(t)\right)\d t=0,
			\end{align}
			where $\mathcal{F}(\bu)$ is defined in \eqref{eqn-F(u)}.
			\item[(iii)] The initial condition is satisfied in the following sense: 
			\begin{align}
				\bu(0) = \bu_0 \ \text{ in }\  \H. 
			\end{align}
		\end{itemize}
	\end{definition}
We now proceed to establish the existence and uniqueness of strong solutions for problem \eqref{eqn-constrained-NSE-N}.
	\begin{theorem}\label{thm-exis-unique}
		For $\bu_0\in\V\cap\mathcal{M}$ and $\U\in\mathrm{L}^2(0,T;\H)$, there exists a \emph{unique strong solution}
		\begin{align*}
			\bu\in\C([0,T];\V)\cap\mathrm{L}^2(0,T;\mathrm{D}(\A))\ \text{ and }\ \frac{\d\bu}{\d t}\in\mathrm{L}^2(0,T;\H)
		\end{align*}
		to the forced constrained 2D NSE \eqref{eqn-constrained-NSE-N} 	such that 
		\begin{align*}
			\bu(t)\in\mathcal{M}\ \text{ for all }\ t\in[0,T].
		\end{align*}
		Moreover, the following energy estimate is satisfied: 
		\begin{align}\label{eqn-main-energy}
			&	\nu\sup_{t\in[0,T]}	\|\nabla\bu(t)\|_{\H}^2+\nu^2\int_0^T\|\A\bu(t)\|_{\H}^2\d t+	\int_0^T\left\|\frac{\d\bu(t)}{\d t}\right\|_{\H}^2\d t
			\nonumber\\&\leq C_{\nu,T}\left[1+\left(\nu\|\nabla\bu_0\|_{\H}^2+\int_0^T\|\U(t)\|_{\H}^2\d t\right)^2\right].
		\end{align}
	\end{theorem}
	\begin{proof}	
		We divide the proof into the following steps:
		\vskip 0.1 cm 
		\noindent 
		\textbf{Step 1:} \emph{Existence:}
		Let $\{\w_1, \w_2, \ldots, \w_m, \ldots\}$ be a complete orthonormal basis of $\H$ contained in $\V$; for example, one may choose the eigenfunctions of the Stokes operator.
		Define the finite-dimensional subspace
		$\H_m = \operatorname{span}\{\w_1, \ldots, \w_m\} \subset \H.$
		Clearly, we have the inclusion $\H_m \subset \H_{m+1} \subset \V$.
		Let $\Pi_m$ denote the orthogonal projection from $\H$ onto $\H_m$, defined by
		$\Pi_m \x = \sum_{i=1}^m ( \x, \w_i ) \w_i.$
		For a given $T > 0$ and for each $m \in \mathbb{N}$, we look for an approximate solution of the form
		$\bu_m(t) = \sum_{k=1}^m g_{km}(t) \w_k,$
		which satisfies the initial value problem associated with the following finite-dimensional system of ordinary differential equations (ODEs) in $\H_m$:
		\begin{equation}\label{eqn-constrained-APPR}
			\left\{
			\begin{aligned}
				\frac{\d\bu_m(t)}{\d t}&=
				-\nu\A\bu_m(t)+\nu\|\nabla\bu_m(t)\|_{\H}^2\bu_m(t)-\Pi_m\B(\bu_m(t),\bu_m(t))+\U_m(t)\\&\quad-(\U_m(t),\bu_m(t))\bu_m(t),\ t\in(0,T), \\
				\bu_m(0)&=\frac{\bu_0^m}{\|\bu_0^m\|_{\H}},
			\end{aligned}
			\right. 
		\end{equation} 
		where $\U_m=\Pi_m\U$, $\bu_0^m=\Pi_m\bu_0$ and $g_{km}(0)=\frac{1}{\|\bu_0^m\|_{\H}}(\bu_0,\w_k)$. Since $\bu_0^m\to \bu_0$, we infer $\|\bu_0^m\|_{\H}\to \|\bu_0\|_{\H}=1$ as $m\to\infty$. Therefore, we have 
		\begin{align}\label{eqn-intial-strong}
			\|\bu_m(0)-\bu_0\|_{\H}&\leq\left\|\frac{\Pi_m\bu_0}{\|\bu_0^m\|_{\H}}-\Pi_m\bu_0\right\|_{\H}+\left\|\Pi_m\bu_0-\bu_0\right\|_{\H}\nonumber\\&\leq\|\bu_0\|_{\H}\left(\frac{1}{\|\bu_0^m\|_{\H}}-1\right)+\left\|\Pi_m\bu_0-\bu_0\right\|_{\H}\nonumber\\&\to 0\ \text{ as }\  m\to\infty. 
		\end{align}
		Note that $(\U_m,\bu_m)=(\Pi_m\U,\bu_m)=(\U,\bu_m)$, since $\Pi_m\bu_m=\bu_m$. Since $\U\in\mathrm{L}^2(0,T;\H)$ and  $\mathcal{F}(\bu)=:-\nu\A\bu+\nu\|\nabla\bu\|_{\H}^2\bu-\B(\bu)+\U-(\U,\bu)\bu$ satisfies a locally Lipschitz condition (see \eqref{eqn-loc-lip}), by Carath\'eodory's existence theorem, we infer the existence and uniqueness of a local solution $\bu_m\in\mathrm{C}([0,T^*_m];\V_m)$ to the problem \eqref{eqn-constrained-APPR} for some $0<T^*_m\leq T$. By showing uniform energy estimates, one can obtain that the time $T^*_m$ can be extended to $T$. 
		
		Taking the inner product with $\bu_m(\cdot)$  to the first equation in \eqref{eqn-constrained-APPR}, we find 
		\begin{align*}
			\frac{\d}{\d t}\left(\|\bu_m(t)\|_{\H}^2-1\right)&=2\left(\frac{\d\bu_m(t)}{\d t},\bu_m(t)\right)\nonumber\\&=2\nu\|\nabla\bu_m(t)\|_{\H}^2\left(\|\bu_m(t)\|_{\H}^2-1\right)-2(\U_m(t),\bu_m(t))\left(\|\bu_m(t)\|_{\H}^2-1\right),
		\end{align*}
		where we have used the fact that $(\B(\bu_m,\bu_m),\bu_m)=0$. The variation of constants formula yields 
		\begin{align}
			\left(\|\bu_m(t)\|_{\H}^2-1\right)=\left(\|\bu_m(0)\|_{\H}^2-1\right)\exp\left\{2\int_0^t\left[\nu\|\nabla\bu_m(s)\|_{\H}^2-(\U_m(s),\bu_m(s))\right]\d s\right\},
		\end{align} 
		for all $t\in[0,T]$. Since $ \|\bu_m(0)\|_{\H}^2=1$, we deduce that $\|\bu_m(t)\|_{\H}^2=1$ all $t\in[0,T]$, provided 
		\begin{align*}
		\int_0^t\left[\nu\|\nabla\bu_m(s)\|_{\H}^2-(\U_m(s),\bu_m(s))\right]\d s\leq \int_0^t\left[\nu\|\nabla\bu_m(s)\|_{\H}^2+\|\U_m(s)\|_{\H}\|\bu_m(s)\|_{\H}\right]\d s<\infty.
		\end{align*} Since $\bu_m\in\mathrm{C}([0,T_m^*];\V_m)$ and all norms in finite-dimensional spaces are equivalent, the finiteness can be shown for any $t\in[0,T_m^*]$. The following energy estimate demonstrates that the time $T_m^*$	can be extended up to $T$. Therefore, we have the following invariance property:
		\begin{align}\label{eqn-est-ener-9}
			\bu_m(t)\in\mathcal{M}\ \text{ for all }\ t\in[0,T]. 
		\end{align}

		Let us now define  the following \emph{energy functional} involving the geometric Lagrange multiplier $(\U_m,\bu_m)$
		\begin{align}
			\mathscr{E}(\bu_m)=\frac{\nu}{2}\int_{\mathbb{T}^2}|\nabla\bu_m(x)|^2\d x-\frac{1}{2}\int_{\mathbb{T}^2}(\U_m(x),\bu_m(x))\d x.
		\end{align}
		Then, we compute $\nabla_{\mathcal{M}}\mathscr{E}(\bu_m)$ as 
		\begin{align}
			\nabla_{\mathcal{M}}\mathscr{E}(\bu_m)&=\Pi_{\bu_m}(\nabla\mathscr{E})=\Pi_{\bu_m}\left(\nu\A\bu_m-\frac{1}{2}\U_m\right)
			\nonumber\\&=\nu\A\bu_m-\nu\|\nabla\bu_m\|_{\H}^2\bu_m-\frac{1}{2}\left(\U_m-(\U_m,\bu_m)\bu_m\right).
		\end{align}
		Moreover, for $\bu_m\in\mathrm{D}(\A)\cap\mathcal{M}$ (sine $\bu_m$ is finite dimensional), we compute 
		\begin{align}\label{eqn-grad-est-10}
			\|\nabla_{\mathcal{M}}\mathscr{E}(\bu_m)\|_{\H}^2&=\left\|\nu\A\bu_m-\nu\|\nabla\bu_m\|_{\H}^2\bu_m-\frac{1}{2}\left(\U_m-(\U_m,\bu_m)\bu_m\right)\right\|_{\H}^2\nonumber\\&=\nu^2\left\|\A\bu_m-\|\nabla\bu_m\|_{\H}^2\bu_m\right\|_{\H}^2+\frac{1}{4}\left\|\U_m-(\U_m,\bu_m)\bu_m\right\|_{\H}^2\nonumber\\&\quad-\nu\left(\A\bu_m-\|\nabla\bu_m\|_{\H}^2\bu_m,\U_m-(\U_m,\bu_m)\bu_m\right)\nonumber\\&=\nu^2\left[\|\A\bu_m\|_{\H}^2+\|\nabla\bu_m\|_{\H}^4\|\bu_m\|_{\H}^2-2\|\nabla\bu_m\|_{\H}^2(\A\bu_m,\bu_m)\right]\nonumber\\&\quad+\frac{1}{4}\left[\|\U_m\|_{\H}^2+|(\U_m,\bu_m)|^2\|\bu_m\|_{\H}^2-2|(\U_m,\bu_m)|^2\right]\nonumber\\&\quad-\nu\left[(\A\bu_m,\U_m)-(\U_m,\bu_m)(\A\bu_m,\bu_m)-\|\nabla\bu_m\|_{\H}^2(\bu_m,\U_m)\right.\nonumber\\&\left.\quad+\|\nabla\bu_m\|_{\H}^2(\U_m,\bu_m)\|\bu_m\|_{\H}^2\right]\nonumber\\&=\nu^2\left[\|\A\bu_m\|_{\H}^2-\|\nabla\bu_m\|_{\H}^4\right]+\frac{1}{4}\left[\|\U_m\|_{\H}^2-|(\U_m,\bu_m)|^2\right]\nonumber\\&\quad-\nu\left[(\A\bu_m,\U_m)-(\U_m,\bu_m)\|\nabla\bu_m\|_{\H}^2\right].
		\end{align}
		
		Taking the inner product with $\nu\A\bu_m(\cdot)$ to the first equation in \eqref{eqn-constrained-APPR}, using  $(\B(\bu_m),\A\bu_m)=0$ and $\bu_m(t)\in\mathcal{M}$ for all $t\in[0,T]$, we obtain 
		\begin{align}\label{eqn-ener-est-0}
			&\frac{\nu}{2}	\frac{\d}{\d t}\|\nabla\bu_m(t)\|_{\H}^2\nonumber\\&=\left(\nu\A\bu_m(t),\frac{\d\bu_m(t)}{\d t}\right)\nonumber\\&= (\nu\A\bu_m(t),-\nu\A\bu_m(t)+\nu\|\nabla\bu_m(t)\|_{\H}^2\bu_m(t)-\B(\bu_m(t),\bu_m(t)))\nonumber\\&\quad +(\nu\A\bu_m(t),\U_m(t)-(\U_m(t),\bu_m(t))\bu_m(t))\nonumber\\&=-\nu^2\left[\|\A\bu_m(t)\|_{\H}^2-\|\nabla\bu_m(t)\|_{\H}^4\right]+\nu\left[(\A\bu_m(t),\U_m(t))-(\U_m(t),\bu_m(t))\|\nabla\bu_m(t)\|_{\H}^2\right]\nonumber\\&= -	\|\nabla_{\mathcal{M}}\mathscr{E}(\bu_m(t))\|_{\H}^2+\frac{1}{4}\left[\|\U_m(t)\|_{\H}^2-|(\U_m(t),\bu_m(t))|^2\right]\nonumber\\&\leq -	\|\nabla_{\mathcal{M}}\mathscr{E}(\bu_m(t))\|_{\H}^2+\frac{1}{2}\|\U_m(t)\|_{\H}^2,
		\end{align} 
	for a.e. $t\in[0,T]$.	Integrating the above inequality from $0$ to $t$ and then taking supremum over time, we find
		\begin{align}\label{eqn-ener-est-1}
			\nu\sup_{t\in[0,T]}	\|\nabla\bu_m(t)\|_{\H}^2+2\int_0^T	\|\nabla_{\mathcal{M}}\mathscr{E}(\bu_m(t))\|_{\H}^2\d t\leq \nu\|\nabla\bu_0\|_{\H}^2+\int_0^T\|\U(t)\|_{\H}^2\d t,
		\end{align}
		for all $t\in[0,T]$. 
		
		From the relation \eqref{eqn-grad-est-10}, it is immediate that 
		\begin{align}
			\nu^2\|\A\bu_m\|_{\H}^2&=	\|\nabla_{\mathcal{M}}\mathscr{E}(\bu_m)\|_{\H}^2+\nu^2\|\nabla\bu_m\|_{\H}^4-\frac{1}{4}\left[\|\U_m\|_{\H}^2-|(\U_m,\bu_m)|^2\right]\nonumber\\&\quad+\nu\left[(\A\bu_m,\U_m)-(\U_m,\bu_m)\|\nabla\bu_m\|_{\H}^2\right]
			\nonumber\\&\leq \|\nabla_{\mathcal{M}}\mathscr{E}(\bu_m)\|_{\H}^2+\nu^2\|\nabla\bu_m\|_{\H}^4+\frac{1}{2}\|\U_m\|_{\H}^2+\nu\|\A\bu_m\|_{\H}\|\U_m\|_{\H}\nonumber\\&\quad+\nu\|\U_m\|_{\H}\|\nabla\bu_m\|_{\H}^2
			\nonumber\\&\leq \|\nabla_{\mathcal{M}}\mathscr{E}(\bu_m)\|_{\H}^2+2\nu^2\|\nabla\bu_m\|_{\H}^4+\frac{\nu^2}{2}\|\A\bu_m\|_{\H}^2+\frac{5}{4}\|\U_m\|_{\H}^2.
		\end{align}
		Therefore, using the estimate \eqref{eqn-ener-est-1}, we deduce 
		\begin{align}\label{eqn-ener-est-2}
			\nu^2\int_0^T\|\A\bu_m(t)\|_{\H}^2\d t&\leq 2 \int_0^T\|\nabla_{\mathcal{M}}\mathscr{E}(\bu_m(t))\|_{\H}^2\d t+4\nu^2\int_0^T\|\nabla\bu_m(t)\|_{\H}^4\d t+\frac{5}{2}\int_0^T\|\U_m(t)\|_{\H}^2\d t\nonumber\\&\leq \nu\|\nabla\bu_0\|_{\H}^2+\int_0^T\|\U_m(t)\|_{\H}^2\d t +4\nu T\left(\nu\|\nabla\bu_0\|_{\H}^2+\int_0^T\|\U_m(t)\|_{\H}^2\d t\right)^2\nonumber\\&\quad+\frac{5}{2}\int_0^T\|\U_m(t)\|_{\H}^2\d t\nonumber\\&\leq C_{\nu,T}\left[1+\left(\nu\|\nabla\bu_0\|_{\H}^2+\int_0^T\|\U(t)\|_{\H}^2\d t\right)^2\right].
		\end{align}

		Using the  fact \eqref{eqn-est-ener-9}, we immediately have $\frac{\d}{\d t}\|\bu_m(t)\|_{\H}^2=2\left(\frac{\d\bu_m(t)}{\d t},\bu_m(t)\right)=0$ for a.e. $t\in[0,T]$, so that 
		\begin{align}\label{eqn-ener-est-3}
			\frac{\nu}{2}	\frac{\d}{\d t}\|\nabla\bu_m(t)\|_{\H}^2&=\left(\nu\A\bu_m(t),\frac{\d\bu_m(t)}{\d t}\right)\nonumber\\&= \left(-\frac{\d\bu_m(t)}{\d t}+\nu\|\nabla\bu_m(t)\|_{\H}^2\bu_m(t)-\B(\bu_m(t),\bu_m(t)),\frac{\d\bu_m(t)}{\d t}\right)\nonumber\\&\quad+\left(\U_m(t)-(\U_m(t),\bu_m(t))\bu_m(t),\frac{\d\bu_m(t)}{\d t}\right)\nonumber\\&=-\left\|\frac{\d\bu_m(t)}{\d t}\right\|_{\H}^2-\left(\B(\bu_m(t),\bu_m(t)),\frac{\d\bu_m(t)}{\d t}\right)+\left(\U_m(t),\frac{\d\bu_m(t)}{\d t}\right)\nonumber\\&\leq-\frac{1}{2}\left\|\frac{\d\bu_m(t)}{\d t}\right\|_{\H}^2+\|\B(\bu_m(t),\bu_m(t))\|_{\H}^2+\|\U_m(t)\|_{\H}^2. 
		\end{align}
		An application of H\"older's, Agmon's and Young's inequalities yield 
		\begin{align}\label{eqn-ener-est-4}
			\|\B(\bu_m,\bu_m)\|_{\H}^2&\leq\|\bu_m\|_{\mathbb{L}^{\infty}}^2\|\nabla\bu_m\|_{\H}^2\leq C\|\bu_m\|_{\H}\|\A\bu_m\|_{\H}\|\nabla\bu_m\|_{\H}^2\nonumber\\&\leq\frac{1}{2}\|\A\bu_m\|_{\H}^2+C\|\nabla\bu_m\|_{\H}^4,
		\end{align}
		where we have used the fact that $\bu_m\in\mathcal{M}$. Therefore, using \eqref{eqn-ener-est-1}, \eqref{eqn-ener-est-2} and \eqref{eqn-ener-est-4} in \eqref{eqn-ener-est-3}, we infer 
		\begin{align}\label{eqn-ener-est-5}
			&	\nu\sup_{t\in[0,T]}\|\nabla\bu_m(t)\|_{\H}^2+\int_0^T\left\|\frac{\d\bu_m(t)}{\d t}\right\|_{\H}^2\d t\nonumber\\&\leq \nu\|\nabla\bu_m(0)\|_{\H}^2+\int_0^T\|\A\bu_m(t)\|_{\H}^2\d t+C\int_0^T\|\nabla\bu_m(t)\|_{\H}^4\d t+2\int_0^T\|\U_m(t)\|_{\H}^2\d t\nonumber\\&\leq C_{\nu,T}\left[1+\left(\nu\|\nabla\bu_0\|_{\H}^2+\int_0^T\|\U(t)\|_{\H}^2\d t\right)^2\right].
		\end{align}
		
		From the energy estimates \eqref{eqn-ener-est-1}, \eqref{eqn-ener-est-2}, and \eqref{eqn-ener-est-5}, together with the Banach-Alaoglu theorem, we can extract a subsequence (again denoted by the same symbol) such that
		\begin{equation}\label{eqn-weak-conv}
			\left\{
			\begin{aligned}
				\bu_m&\xrightarrow{w^*}\bu \ \text{ in }\ \mathrm{L}^{\infty}(0,T;\V),\\
				\bu_m&\xrightarrow{w}\bu \ \text{ in }\ \mathrm{L}^{2}(0,T;\mathrm{D}(\A)),\\
				\frac{\d\bu_m}{\d t}&\xrightarrow{w}\frac{\d\bu}{\d t} \ \text{ in }\ \mathrm{L}^{2}(0,T;\H).
			\end{aligned}
			\right. 
		\end{equation}
		Using the weakly lowersemicontinuity of the norms, we deduce 
			\begin{align}\label{eqn-ener-est-11}
			\|\bu(t)\|_{\H}^2\leq\liminf_{m\to\infty}\|\bu_m(t)\|_{\H}^2=1\ \text{ for all }\ t\in[0,T],
		\end{align}
		The above convergences, combined with the Aubin-Lions-Simon compactness lemma (\cite[Theorem 5, Section 8]{SJ-1987}), also lead to the following additional convergences:
		\begin{equation}\label{eqn-strong-conv}
			\left\{
			\begin{aligned}
				\bu_m&\to\bu \ \text{ in }\ \mathrm{L}^{2}(0,T;\V),\\
				\bu_m&\to\bu \ \text{ in }\ \mathrm{C}([0,T];\H). 
			\end{aligned}
			\right. 
		\end{equation}
		The convergences of the linear terms in \eqref{eqn-constrained-APPR} are clear. We need to justify the convergences of nonlinear terms only. Applying  H\"older's and Agmon's inequalities, using  the energy estimates \eqref{eqn-ener-est-1}, \eqref{eqn-ener-est-2}, and \eqref{eqn-ener-est-5}, and the convergences given in \eqref{eqn-strong-conv}, we obtain 
		\begin{align}\label{eqn-bilinear-conv}
			&\int_0^T\|\B(\bu_m(t),\bu_m(t))-\B(\bu(t),\bu(t))\|_{\H}^2\d t\nonumber\\&\leq 2\int_0^T\|\B(\bu_m(t)-\bu(t),\bu_m(t))\|_{\H}^2\d t+2\int_0^T\|\B(\bu(t),\bu_m(t)-\bu(t))\|_{\H}^2\d t\nonumber\\&\leq 2\int_0^T\|\bu_m(t)-\bu(t)\|_{\mathbb{L}^{\infty}}^2\|\bu_m(t)\|_{\V}^2\d t+2\int_0^T\|\bu(t)\|_{\mathbb{L}^{\infty}}^2\|\nabla(\bu_m(t)-\bu(t))\|_{\mathbb{L}^2}^2\d t\nonumber\\&\leq CT^{1/2}\sup_{t\in[0,T]}\|\bu_m(t)-\bu(t)\|_{\H}\sup_{t\in[0,T]}\|\bu_m(t)\|_{\V}^2\left(\int_0^T\|\A\bu_m(t)\|_{\H}^2\d t\right)^{1/2}\nonumber\\&\quad+C\sup_{t\in[0,T]}\|\bu(t)\|_{\H}\sup_{t\in[0,T]}\left(\|\bu_m(t)\|_{\V}+\|\bu(t)\|_{\V}\right)\left(\int_0^T\|\A\bu(t)\|_{\H}^2\d t\right)^{1/2}\nonumber\\&\qquad\times\left(\int_0^T\|\bu_m(t)-\bu(t)\|_{\V}^2\d t\right)^{1/2}\nonumber\\&\to 0\ \text{ as }\ m\to\infty. 
		\end{align}
		Using  the fact that $\bu_m(t)\in\mathcal{M}$ for all $t\in[0,T]$, the energy estimates \eqref{eqn-ener-est-1} and  and the convergences given in \eqref{eqn-strong-conv}, we find 
		\begin{align}\label{eqn-nonlinear-conv}
			&	\int_0^T\|\|\nabla\bu_m(t)\|_{\H}^2\bu_m(t)-\|\nabla\bu(t)\|_{\H}^2\bu(t)\|_{\H}^2\d t\nonumber\\&\leq 2\int_0^T\left(\|\nabla\bu_m(t)\|_{\H}^2-\|\nabla\bu(t)\|_{\H}^2\right)^2\|\bu_m(t)\|_{\H}^2\d t+2\int_0^T\|\nabla\bu(t)\|_{\H}^4\|\bu_m(t)-\bu(t)\|_{\H}^2\d t\nonumber\\&\leq 4\sup_{t\in[0,T]}\left(\|\bu_m(t)\|_{\V}^2+\|\bu(t)\|_{\V}^2\right)\int_0^T\|\bu_m(t)-\bu(t)\|_{\V}^2\d t\nonumber\\&\quad+2T\sup_{t\in[0,T]}\|\bu(t)\|_{\V}^4\sup_{t\in[0,T]}\|\bu_m(t)-\bu(t)\|_{\H}^2\nonumber\\&\to 0\ \text{ as }\ m\to\infty. 
		\end{align}
		Finally, for any $\boldsymbol{\varphi}\in\mathrm{C}([0,T];\V_n)$, for $n<m$, we consider 
		\begin{align}\label{eqn-conv}
			&	\left|\int_0^T((\U_m(t)-(\U_m(t),\bu_m(t))\bu_m(t))-(\U_m(t)-(\U_m(t),\bu(t))\bu(t)),\boldsymbol{\varphi}(t))\d t\right|\nonumber\\&\leq\left|\int_0^T(\U_m(t)-\U_m(t),\boldsymbol{\varphi}(t))\d t\right|\nonumber\\&\quad+ \left|\int_0^T(\U_m(t),\bu_m(t))(\bu_m(t)-\bu(t),\boldsymbol{\varphi}(t))\d t\right|+\left|\int_0^T(\U_m(t),\bu_m(t)-\bu(t))(\bu(t),\boldsymbol{\varphi}(t))\d t\right|
			\nonumber\\&\quad+\left|\int_0^T(\U_m(t)-\U_m(t),\bu(t))(\bu(t),\boldsymbol{\varphi}(t))\d t\right|\nonumber\\&\leq \left|\int_0^T(\U_m(t)-\U_m(t),\boldsymbol{\varphi}(t))\d t\right|+\sup_{t\in[0,T]}\|\boldsymbol{\varphi}(t)\|_{\H}\int_0^T|(\U_m(t)-\U_m(t),\bu(t))|\d t\nonumber\\&\quad+2\sup_{t\in[0,T]}\|\bu_m(t)-\bu(t)\|_{\H}\left(\int_0^T\|\U_m(t)\|_{\H}^2\d t\right)^{1/2}\left(\int_0^T\|\boldsymbol{\varphi}(t)\|_{\H}^2\d t\right)^{1/2}\nonumber\\&=:I_1^m+I_2^m+I_3^m,
		\end{align}
		where we have used the fact that $\bu_m(t)\in\mathcal{M}$ for all $t\in[0,T]$ and the estimate \eqref{eqn-ener-est-11}. Note that $I_1^m,I_2^m\to 0$ as $m\to\infty$, since $\U_m,\U$ in $\mathrm{L}^2(0,T;\H)$. By using the convergence given in \eqref{eqn-strong-conv} and the fact that $\U\in\mathrm{L}^2(0,T;\H)$, we immediately  have $I_3^m\to 0$ as $m\to\infty$. Note that the above convergence holds true for all $n\in\mathbb{N}$. Since $\mathrm{C}([0,T];\V_n)$ is dense in $\mathrm{L}^2(0,T;\H)$, the  convergence \eqref{eqn-conv} holds true for all $\boldsymbol{\varphi}\in\mathrm{L}^2(0,T;\H)$. 
		
		Therefore, we can pass the limit in the equation \eqref{eqn-constrained-APPR} and obtain the existence of a strong solution $\bu$ to the problem \eqref{eqn-constrained-NSE-N}. Using the weakly lowersemicontinuity of the norms, we infer 
		\begin{align}\label{eqn-ener-est-12}
			&	\nu\sup_{t\in[0,T]}	\|\nabla\bu(t)\|_{\H}^2+\nu^2\int_0^T\|\A\bu(t)\|_{\H}^2\d t+	\int_0^T\left\|\frac{\d\bu(t)}{\d t}\right\|_{\H}^2\d t\nonumber\\&\leq 
			\liminf_{m\to\infty}	\nu\sup_{t\in[0,T]}	\|\nabla\bu_m(t)\|_{\H}^2+\nu^2	\liminf_{m\to\infty}\int_0^T\|\A\bu_m(t)\|_{\H}^2\d t+	\liminf_{m\to\infty}\int_0^T\left\|\frac{\d\bu_m(t)}{\d t}\right\|_{\H}^2\d t
			\nonumber\\&\leq \liminf_{m\to\infty}\left[	\nu\sup_{t\in[0,T]}	\|\nabla\bu_m(t)\|_{\H}^2+\nu^2	\int_0^T\|\A\bu_m(t)\|_{\H}^2\d t+\int_0^T\left\|\frac{\d\bu_m(t)}{\d t}\right\|_{\H}^2\d t\right]
			\nonumber\\&\leq C_{\nu,T}\left[1+\left(\nu\|\nabla\bu_0\|_{\H}^2+\int_0^T\|\U(t)\|_{\H}^2\d t\right)^2\right],
		\end{align}
		which completes the proof of \eqref{eqn-main-energy}. 	We immediately have the initial data convergence $\bu(t)\to\bu_0$ as $t\to 0$ in $\H$, since  for any $t>0$
		\begin{align}
			\|\bu(t)-\bu_0\|_{\H}\leq 	\|\bu(t)-\bu_m(t)\|_{\H}+\|\bu_m(t)-\bu_m(0)\|_{\H}+\|\bu_m(0)-\bu_0\|_{\H}. 
		\end{align} 
		For fixed large $m$, continuity of $\bu_m$ gives the middle term tends to zero  as $t\to 0$, the  first term goes to zero as $m\to\infty$, since  $\bu_m(t)\to\bu(t)$ in $\H$ for all $t\in[0,T]$ (see \eqref{eqn-strong-conv}), 	and the final convergence holds true for sufficiently large $m$ (see \eqref{eqn-intial-strong}). Since $\bu\in\mathrm{L}^{\infty}(0,T;\V)\cap\mathrm{L}^2(0,T;\mathrm{D}(\A))$ and $\frac{\d\bu}{\d t}\in\mathrm{L}^2(0,T;\H)$, an application of the Lions-Magenes lemma (\cite[Lemma 1.2, Chapter 3]{TR-1984}) yields the absolute continuity of the norm $[0,T]\ni t\mapsto\|\nabla\bu(t)\|_{\H}^2$ and the following energy equality for every $t\in[0,T]$: 
		\begin{align}
			&\|\nabla\bu(t)\|_{\H}^2+2\nu\int_0^t\|\A\bu(s)\|_{\H}^2\d s\nonumber\\&=\|\nabla\bu_0\|_{\H}^2+\nu\int_0^t\|\nabla\bu(s)\|_{\H}^4\d s+\int_0^t(\U(s),\A\bu(s))\d s-\int_0^t(\U(s),\bu(s))\|\nabla\bu(s)\|_{\H}^2\d s.
		\end{align}
		Therefore, Radon-Riesz property yields $\bu\in\C([0,T];\V)$ also.

		\vskip 0.1 cm 
		\noindent 
		\textbf{Step 2:} 
			\emph{Invariance of the manifold:} 	Taking the inner product with $\bu(\cdot)$  to the first equation in \eqref{eqn-constrained-NSE-N}, we find 
		\begin{align}
			\frac{\d}{\d t}\left(\|\bu(t)\|_{\H}^2-1\right)&=2\left(\frac{\d\bu(t)}{\d t},\bu(t)\right)\nonumber\\&=2\nu\|\nabla\bu(t)\|_{\H}^2\left(\|\bu(t)\|_{\H}^2-1\right)-2(\U(t),\bu(t))\left(\|\bu(t)\|_{\H}^2-1\right),
		\end{align}
		where we have used the fact that $(\B(\bu,\bu),\bu)=0$. The variation of constants formula yields 
		\begin{align}
			\left(\|\bu(t)\|_{\H}^2-1\right)=\left(\|\bu_0\|_{\H}^2-1\right)\exp\left\{2\int_0^t\left[\nu\|\nabla\bu(s)\|_{\H}^2-(\U(s),\bu(s))\right]\d s\right\},
		\end{align} 
		for all $t\in[0,T]$. Since $\bu_0\in\mathcal{M}$, we deduce that $\|\bu(t)\|_{\H}^2=1$ all $t\in[0,T]$, since 
		\begin{align*}
			\int_0^t\left[\nu\|\nabla\bu(s)\|_{\H}^2-(\U(s),\bu(s))\right]\d s\leq \nu T\sup_{t\in[0,T]}\|\nabla\bu(t)\|_{\H}^2+T^{1/2}\left(\int_0^T\|\U(t)\|_{\H}^2\d t\right)^{1/2},
		\end{align*}
		where we have used \eqref{eqn-ener-est-11} also. This completes the proof of invariance of the manifold.

		\vskip 0.1 cm 
		\noindent 
		\textbf{Step 3:} 
		\emph{Uniqueness:} Let us now prove the uniqueness of strong solutions for the problem \eqref{eqn-constrained-NSE-N}. Let $\bu_1,\bu_2\in \C([0,T];\V\cap\mathcal{M})\cap\mathrm{L}^2(0,T;\mathrm{D}(\A))$ be two strong solutions of \eqref{eqn-constrained-NSE-N} with the same initial data $\bu_0$ and forcing $\U$. Then $\z=\bu_1-\bu_2$ satisfies the following: 
		\begin{equation}\label{eqn-constrained-NSE-D}
			\left\{
			\begin{aligned}
				\frac{\d\bz(t)}{\d t}&=-\nu\A\bz(t)+\nu\|\nabla\bu_1(t)\|_{\H}^2\bz(t)+\left(\|\nabla\bu_1(t)\|_{\H}^2-\|\nabla\bu_2(t)\|_{\H}^2\right)\bu_2(t)\\&\quad-\B(\bu_1(t),\z(t))-\B(\z(t),\bu_2(t))-(\U(t),\bu_1(t))\z(t)-(\U(t),\z(t))\bu_2(t),\ t>0, \\
				\bz(0)&=\boldsymbol{0}.
			\end{aligned}
			\right. 
		\end{equation} 
		Taking the inner product with $\z(\cdot)$ in \eqref{eqn-constrained-NSE-D} and using the fact that $(\B(\bu_1,\z),\z)=0$, we deduce 
		\begin{align}\label{eqn-difference-12}
			&\frac{1}{2}\frac{\d}{\d t}\|\z(t)\|_{\H}^2+\nu\|\nabla\z(t)\|_{\H}^2\nonumber\\&=\nu\|\nabla\bu_1(t)\|_{\H}^2\|\z(t)\|_{\H}^2+\nu\left(\|\nabla\bu_1(t)\|_{\H}^2-\|\nabla\bu_2(t)\|_{\H}^2\right)(\bu_2(t),\z(t))-(\B(\z(t),\bu_2(t)),\z(t))\nonumber\\&\quad-(\U(t),\bu_1(t))\|\z(t)\|_{\H}^2-(\U(t),\z(t))(\bu_2(t),\z(t)),
		\end{align}
	for a.e. $t\in[0,T]$. Using the Cauchy-Schwarz, triangle,  and Young inequalities, we estimate $\nu\left(\|\nabla\bu_1\|_{\H}^2-\|\nabla\bu_2\|_{\H}^2\right)(\bu_2,\z)$ as 
		\begin{align}\label{eqn-difference-13}
			\nu	\left(\|\nabla\bu_1\|_{\H}^2-\|\nabla\bu_2\|_{\H}^2\right)(\bu_2,\z)&\leq 	\nu\left(\|\nabla\bu_1\|_{\H}+\|\nabla\bu_2\|_{\H}\right)\|\nabla(\bu_1-\bu_2)\|_{\H}\|\z\|_{\H}\nonumber\\&\leq\frac{\nu}{4}\|\nabla\z\|_{\H}^2+\nu\left(\|\nabla\bu_1\|_{\H}+\|\nabla\bu_2\|_{\H}\right)^2\|\z\|_{\H}^2.
		\end{align}
		Using H\"older's, Ladyzhenskaya's and Young's inequalities, we estimate $-(\B(\z,\bu_2),\z)$ as 
		\begin{align}\label{eqn-difference-14}
			-(\B(\z,\bu_2),\z)&\leq\|\nabla\bu_2\|_{\H}\|\z\|_{\mathbb{L}^4}^2\leq\sqrt{2}\|\nabla\bu_2\|_{\H}\|\nabla\z\|_{\H}\|\z\|_{\H}\leq\frac{\nu}{4}\|\nabla\z\|_{\H}^2+\frac{1}{\nu}\|\nabla\bu_2\|_{\H}^2\|\z\|_{\H}^2. 
		\end{align}
		Finally, using the fact that $\bu_1(t),\bu_2(t)\in\mathcal{M}$ for all $t\in[0,T]$, and the Cauchy-Schwarz inequality, we estimate the term $-(\U,\bu_1)\|\z\|_{\H}^2-(\U,\z)(\bu_2,\z)$   as
		\begin{align}\label{eqn-difference-15}
			&-(\U,\bu_1)\|\z\|_{\H}^2-(\U,\z)(\bu_2,\z)\leq 2\|\U\|_{\H}\|\z\|_{\H}^2. 
		\end{align}
		Using \eqref{eqn-difference-13}-\eqref{eqn-difference-15} in \eqref{eqn-difference-12}, we find 
		\begin{align}\label{eqn-difference-16}
			&	\|\z(t)\|_{\H}^2+\nu\int_0^t\|\nabla\z(s)\|_{\H}^2\d s\nonumber\\&\leq\|\z(0)\|_{\H}^2+2\int_0^t\left[3\nu\|\nabla\bu_1(s)\|_{\H}^2+\left(3\nu+\frac{1}{\nu}\right)\|\nabla\bu_2(s)\|_{\H}^2+2\|\U(s)\|_{\H}\right]\|\z(s)\|_{\H}^2\d s,
		\end{align}
		for all $t\in[0,T]$. An application of Gr\"onwall's inequality in \eqref{eqn-difference-16} yields 
		\begin{align}
			&	\|\z(t)\|_{\H}^2+\nu\int_0^t\|\nabla\z(s)\|_{\H}^2\d s\nonumber\\&\leq \|\z(0)\|_{\H}^2\exp\left\{2\int_0^T\left[3\nu\|\nabla\bu_1(t)\|_{\H}^2+\left(3\nu+\frac{1}{\nu}\right)\|\nabla\bu_2(t)\|_{\H}^2+2\|\U(t)\|_{\H}\right]\d t\right\}.
		\end{align}
		Since $\z(0)=\boldsymbol{0}$, the uniqueness follows. 
	\end{proof}
	
	\begin{remark}
		Once can show the invariance of the manifold in the following way also.	Since $\bu_m(t)\in\mathcal{M}$ for all $t\in[0,T]$, using the strong convergence $\mathrm{C}([0,T];\H)$ (see \eqref{eqn-strong-conv}), we deduce 
		\begin{align}
			\|\bu(t)\|_{\H}=\lim_{m\to\infty}\|\bu_m(t)\|_{\H}=1, \ \text{ for all }\ t\in[0,T].
		\end{align}
	\end{remark}
	
	\section{Optimal Control Problem}\label{sec3}\setcounter{equation}{0}
	
	The main objective of this section is to prove the existence of an optimal control that minimizes the following  cost functional:
	\begin{equation}\label{cost}
		\begin{aligned}
			\mathscr{J}(\bu,\U) &:= \frac{1}{2} \int_0^T\|\bu(t)-\bu_d\|_{\H}^2 \d t+\frac{1}{2}\int_0^T\|\nabla\times\bu(t)\|_{\H}^2\d t+ \int_0^T\|\U(t)\|_{\H}^2\d t\nonumber\\&\quad+\frac{1}{2}\|\bu(T)-\bu^f\|_{\H}^2,
		\end{aligned}
	\end{equation}
	subject to the constraint  \eqref{eqn-constrained-NSE-N}. 
	In \eqref{cost}, the function $\bu_d \in \mathrm{L}^2(0,T;\H)$ denotes the desired (or reference) velocity field, while $\bu^f \in \H$ represents the target velocity at the final time $T$. From a physical standpoint, our objective is to determine the control $\U$ such that the resulting velocity field $\bu(\cdot)$ remains as close as possible, in the sense made precise by \eqref{control problem}, to the desired velocity $\bu_d$, while simultaneously minimizing turbulence (i.e., minimizing enstrophy).
	Since $\bu(\cdot)$ is divergence-free and satisfies the Dirichlet boundary condition, we have the identity
$\|\mathrm{curl \ }\bu\|_{\H}=\|\nabla\times\bu\|_{\H}=\|\nabla\bu\|_{\H}=\|\A^{1/2}\bu\|_{\H}$.
	Hence, in the remainder of the paper, we replace $\|\nabla\times\bu\|_{\H}$ with $\|\nabla\bu\|_{\H}$ in \eqref{cost}. We define $$\mathrm{W}^{1,2}(0,T;\mathbb{X},\mathbb{Y}):=\left\{\bu\in\mathrm{L}^2(0,T;\mathbb{X}):\frac{\d\bu}{\d t}\in\mathrm{L}^2(0,T;\mathbb{Y})\right\}.$$
	We observe by the Lions-Magenes lemma that if $\mathbb{X}=\V$ and $\mathbb{Y}=\V^{\prime}$, then $ \mathrm{W}^{1,2}(0,T;\V,\V^{\prime})\hookrightarrow\C([0,T];\H)$ and if $\mathbb{X}=\mathrm{D}(\A)$ and $\mathbb{Y}=\H$, then   $ \mathrm{W}^{1,2}(0,T;\mathrm{D}(\A),\H)\hookrightarrow\C([0,T];\V)$.

	\begin{definition}[Admissible class]\label{definition 1}
		The \emph{admissible class} $\mathscr{A}_{\mathrm{ad}}$ of pairs $$(\bu,\U)\in\mathrm{W}^{1,2}(0,T;\mathrm{D}(\A),\H)\cap\mathcal{M}\times\mathrm{L}^{2}(0,T;\mathbb{H})=:\mathcal{X} $$ is defined as the set of states $\bu(\cdot)$ solving the system \eqref{eqn-constrained-NSE-N} with the control $\U \in \mathscr{U}_{\mathrm{ad}}:=\mathrm{L}^{2}(0,T;\mathbb{H})$. That is,
		\begin{align*}
			\mathscr{A}_{\mathrm{ad}}&:=\big\{(\bu,\U) :\bu\in\mathrm{W}^{1,2}(0,T;\mathrm{D}(\A),\H)\cap\mathcal{M} \text{ is \text{the unique strong  solution} of }\eqref{eqn-constrained-NSE-N} \nonumber\\&\qquad \text{ with the control }\U\in\mathrm{L}^{2}(0,T;\mathbb{H})\big\}.
		\end{align*}
	\end{definition}
We work in the product space
$\mathcal{X},$
endowed with the natural weak topology. It is clear that $\mathscr{A}_{\mathrm{ad}}$ is nonempty, since for every $\U \in \mathscr{U}_{\mathrm{ad}}$, there exists a \emph{unique strong solution} to the system \eqref{eqn-constrained-NSE-N}.	Note that the admissible set $\mathscr{A}_{\mathrm{ad}}$ is \emph{not convex,}  but is \emph{sequentially weakly closed}.
	Let
	$(\bu^n,\U^n)\subset\mathscr A_{\mathrm{ad}}$ be a sequence of admissible
	pairs and suppose
	\[
	(\bu^n,\U^n)\xrightarrow{w} (\bu^*,\U^*) \ \text{weakly in }\ \mathcal{X}.
	\]
Then one can show that  $(\bu^*,\U^*)\in\mathscr A_{\mathrm{ad}}$, that is, \(\bu^*\) is the
	(unique) strong solution of the constrained state equation driven by \(\U^*\)
	and \(\|\bu^*(t)\|_{\H}=1\) for all \(t\) (see Proof of Theorem \ref{thm-exis-optimal} below).
	Accordingly, the optimal control problem under consideration can be formulated as
	\begin{align}\label{control problem}
		\min_{(\bu,\U) \in \mathscr{A}_{\mathrm{ad}}} \mathscr{J}(\bu,\U).
	\end{align}
	A solution to \eqref{control problem} is referred to as an \emph{optimal solution}. The corresponding pair $(\bu^*, \U^*)$ is called the \emph{optimal pair}, and the control $\U^*$ is termed the \emph{optimal control}.
	
	Let us now show that the problem \eqref{control problem} has at least one solution. 
	\begin{theorem}\label{thm-exis-optimal}
		The problem \eqref{control problem} has at least one solution $(\bu^*,\U^*)\in\mathscr{A}_{\mathrm{ad}}$. 
	\end{theorem}
	\begin{proof}
		Let $(\bu^n,\U^n)\in\mathscr{A}_{\mathrm{ad}}$ be a minimizing sequence of the problem \eqref{control problem}, that is, 
		\begin{align}
			\inf(\text{P})\leq\mathscr{J}(\bu^n,\U^n)\leq \inf(\text{P})+\frac{1}{n},
		\end{align}
		where 
	\begin{equation}\label{eqn-constrained-NSE}
		\left\{
		\begin{aligned}
			\frac{\d\bu^n(t)}{\d t}&=-\nu\A\bu^n(t)+\nu\|\nabla\bu^n(t)\|_{\H}^2\bu^n(t)-\B(\bu^n(t),\bu^n(t))\\&\quad+\U^n(t)-(\U^n(t),\bu^n(t))\bu^n(t),\ t\in(0,T), \\
			\bu^n(0)&=\bu_0,
		\end{aligned}
		\right. 
	\end{equation} 
	where $\bu_0\in\V\cap\mathcal{M}$ and $\U^n\in\mathrm{L}^2(0,T;\H)$ is the control acting on the system. 

	Using similar calculations as the energy estimates \eqref{eqn-ener-est-1}, \eqref{eqn-ener-est-2}, and \eqref{eqn-ener-est-5}, together with the Banach-Alaoglu theorem, we can extract a subsequence (again denoted by the same symbol) such that
	\begin{equation}\label{eqn-weak-conv-1}
		\left\{
		\begin{aligned}
			\bu^n&\xrightarrow{w^*}\bu^* \ \text{ in }\ \mathrm{L}^{\infty}(0,T;\V),\\
			\bu^n&\xrightarrow{w}\bu^* \ \text{ in }\ \mathrm{L}^{2}(0,T;\mathrm{D}(\A)),\\
			\frac{\d\bu^n}{\d t}&\xrightarrow{w}\frac{\d\bu^*}{\d t} \ \text{ in }\ \mathrm{L}^{2}(0,T;\H),\\
			\U^n&\xrightarrow{w}\U^*\ \text{ in }\ \mathrm{L}^2(0,T;\H). 
		\end{aligned}
		\right. 
	\end{equation}
	The above convergences, combined with the Aubin-Lions lemma, also lead to the following additional convergences:
	\begin{equation}\label{eqn-strong-conv-1}
		\left\{
		\begin{aligned}
			\bu^n&\to\bu^* \ \text{ in }\ \mathrm{L}^{2}(0,T;\V),\\
			\bu^n&\to\bu^* \ \text{ in }\ \mathrm{C}([0,T];\H). 
		\end{aligned}
		\right. 
	\end{equation}
	Since $\bu^n(t)\in\mathcal{M}$ for all $t\in[0,T]$, using the strong convergence in $\mathrm{C}([0,T];\H)$, we infer
	\begin{align}\label{eqn-est-man-1}
		\|\bu^*(t)\|_{\H}=\lim_{n\to\infty}\|\bu^n(t)\|_{\H}=1, \ \text{ for all }\ t\in[0,T].
	\end{align}
	The convergences of the linear, bilinear and nonlinear terms are clear from \eqref{eqn-weak-conv-1}, \eqref{eqn-bilinear-conv}, \eqref{eqn-nonlinear-conv}, respectively. 
	We need to justify the convergences of the terms involving control terms only. 
	The convergence of the geometric coupling scalar \((\U^n,\bu^n)\) can be shown in the following way.  Note that
	\(\bu^n\to\bu^*\) strongly in \(\mathrm{C}([0,T];\H)\) and \(\U^n\xrightarrow{w}\U^*\)
	in \(\mathrm{L}^2(0,T;\H)\). Hence for any test function \(\psi\in \mathrm{C}_c^\infty((0,T))\),
	\[
	\int_0^T (\U^n(t),\bu^n(t))\,\psi(t)\d t
	= \int_0^T (\U^n(t),\bu^*(t))\psi(t)\d t
	+ \int_0^T (\U^n(t),\bu^n(t)-\bu^*(t))\psi(t)\d t.
	\]
	The first integral tends to \(\int_0^T (\U^*(t),\bu^*(t))\psi(t)\d t\) by
	weak convergence of \(\U_n\) and fixed \(\bu^*\); the second term tends to
	zero because \(\U_n\) is uniformly bounded in \(\mathrm{L}^2(0,T;\H)\) and
	\(\bu_n\to\bu^*\) strongly in \(\C([0,T];\H)\). Thus
	\((\U^n,\bu^n)\xrightarrow{w} (\U^*,\bu^*)\) in the sense of distributions. Since $\mathrm{C}_c^\infty((0,T))$ is dense in $\mathrm{L}^2(0,T)$, the above convergence holds true for all $\psi\in\mathrm{L}^2(0,T)$. In particular $(\bu^*,\boldsymbol{\phi})\in\mathrm{L}^2(0,T)$. 
	Hence for all $\boldsymbol{\phi}\in\mathrm{L}^2(0,T;\H)$, we find 
	\begin{align*}
		&	\int_0^T\left((\U^n(t),\bu^n(t))\bu^n(t),\boldsymbol{\phi}(t)\right)\d t\nonumber\\&= 	\int_0^T(\U^n(t),\bu^n(t))\left(\bu^n(t)-\bu^*(t),\boldsymbol{\phi}(t)\right)\d t\nonumber\\&\quad+ \int_0^T\left[(\U^n(t),\bu^n(t))-(\U^*(t),\bu^*(t))\right]\left(\bu^*(t),\boldsymbol{\phi}(t)\right)\d t\nonumber\\&\leq\sup_{t\in[0,T]}\|\bu^n(t)-\bu^*(t)\|_{\H}\left(\int_0^T\|\U^n(t)\|_{\H}^2\d t\right)^{1/2}\left(\int_0^T\|\boldsymbol{\phi}(t)\|_{\H}^2\d t\right)^{1/2} \nonumber\\&\quad+ \int_0^T\left[(\U^n(t),\bu^n(t))-(\U^*(t),\bu^*(t))\right]\left(\bu^*(t),\boldsymbol{\phi}(t)\right)\d t,
	\end{align*}
	where we have used the fact that $\bu^n,\bu^*\in\mathcal{M}$. 	Therefore, we finally have 
	\[(\U^n,\bu^n)\bu^n \xrightarrow{w} (\U^*,\bu^*)\bu^*\ \text{ in }\ 
	\mathrm{L}^2(0,T;\H).\]

By passing to the limit in equation \eqref{eqn-constrained-NSE}, we establish the existence of a strong solution $\bu^*$ to problem \eqref{eqn-constrained-NSE}. Since $\bu^*$ is the unique strong solution of \eqref{eqn-constrained-NSE-N} corresponding to the control $\U^* \in \mathrm{L}^2(0,T; \mathbb{H})$, it follows that the entire sequence $\{\bu_n\}$ converges to $\bu^*$. Consequently, we have $(\bu^*, \U^*) \in \mathscr{A}_{\mathrm{ad}}$.

We now demonstrate that $(\bu^*, \U^*)$ is indeed a minimizer, that is,
$
\mathscr{J} = \mathscr{J}(\bu^*, \U^*).
$
Since the cost functional $\mathscr{J}(\cdot, \cdot)$ is continuous and convex on $\mathrm{L}^2(0,T; \V) \times \mathscr{U}_{\mathrm{ad}}$ (see Propositions III.1.6 and III.1.10 in \cite{EI+TT-83}), it is weakly lower semicontinuous (cf. Proposition II.4.5 in \cite{EI+TT-83}).
Therefore, for any sequence
$
(\bu_n, \U_n) \xrightarrow{w} (\bu^*, \U^*) \  \text{ in} \ \mathrm{L}^2(0,T; \V) \times \mathrm{L}^2(0,T; \mathbb{U}),
$
we have
$
\mathscr{J}(\bu^*, \U^*) \le \liminf_{n \to \infty} \mathscr{J}(\bu_n, \U_n).
$
Consequently,
$$
\mathscr{J} \le \mathscr{J}(\bu^*, \U^*) \le \liminf_{n \to \infty} \mathscr{J}(\bu_n, \U_n)
= \lim_{n \to \infty} \mathscr{J}(\bu_n, \U_n) = \mathscr{J},
$$
which shows that $(\bu^*, \U^*) \in \mathscr{A}_{\mathrm{ad}}$ is indeed a minimizer of the control problem \eqref{control problem}.
	\end{proof}

	\subsection{The linearized system} 
In this subsection, we analyze the linearized system and demonstrate its global well-posedness under the assumption that $\bu(\cdot)$ is a \emph{strong solution} of system \eqref{eqn-constrained-NSE-N} corresponding to the control $\U(\cdot)$, and that $\W \in \mathrm{L}^2(0,T;\H)$.
Subsequently, we establish that when $\W$ is chosen as $\U-(\U,\bu_{\U^*})\bu_{\U^*}$, the corresponding  function $\boldsymbol{\eta}$ lies within the tangent cone of admissible perturbations (see Proposition \ref{lem3.8} below). We consider the following linearized system: 
	\begin{equation}\label{4}
		\left\{
		\begin{aligned}
			\frac{\d\boldsymbol{\eta}(t)}{\d t}&=-\nu\A\boldsymbol{\eta}(t)+\nu\|\nabla\bu(t)\|_{\H}^2\boldsymbol{\eta}(t)+2\nu(\nabla\bu(t),\nabla\boldsymbol{\eta}(t))\bu(t)-\B'(\bu(t))\boldsymbol{\eta}(t)\\&\quad+\W(t)-[(\U(t),\bu(t))\boldsymbol{\eta}(t)+(\U(t),\boldsymbol{\eta}(t))\bu(t)],\ \text{ in }\ \V',\\
			\boldsymbol{\eta}(0)&=\boldsymbol{\eta}_0,
		\end{aligned}
		\right.
	\end{equation}
	for a.e. $t\in[0,T]$,  where  $\B'(\bu)\boldsymbol{\eta}=\B(\bu,\boldsymbol{\eta})+\B(\boldsymbol{\eta},\bu)$ and  $\boldsymbol{\eta}_0\in\H$.  We begin by deriving a formal \emph{a priori} energy estimate satisfied by system \eqref{4}. A rigorous justification of this estimate can be obtained using the standard Faedo-Galerkin approximation method. Taking the inner product of the first equation in \eqref{4} with $\boldsymbol{\eta}(\cdot)$, using  the fact that $\langle\B(\bu,\boldsymbol{\eta}),\boldsymbol{\eta}\rangle=0$, $\bu(t)\in\mathcal{M}$ for all $t\in[0,T]$, applying H\"older's, Ladyzhenskaya's and Young's inequalities, we obtain 
		\begin{align}
		&\frac{1}{2}\frac{\d}{\d t}\|\boldsymbol{\eta}(t)\|_{\H}^2+\nu\|\nabla\boldsymbol{\eta}(t)\|_{\H}^2
		\nonumber\\&=\nu\|\nabla\bu(t)\|_{\H}^2\|\boldsymbol{\eta}(t)\|_{\H}^2+2\nu (\nabla\bu(t),\nabla\boldsymbol{\eta}(t))(\bu(t),\boldsymbol{\eta}(t))-\langle\B(\boldsymbol{\eta}(t),\bu(t)),\boldsymbol{\eta}(t)\rangle\nonumber\\&\quad+(\W(t),\boldsymbol{\eta}(t))-(\U(t),\bu(t))\|\boldsymbol{\eta}(t)\|_{\H}^2-(\U(t),\boldsymbol{\eta}(t))(\bu(t),\boldsymbol{\eta}(t))\nonumber\\&\leq \frac{\nu}{4} \|\nabla\boldsymbol{\eta}(t)\|_{\H}^2+5\nu\|\nabla\bu(t)\|_{\H}^2\|\boldsymbol{\eta}(t)\|_{\H}^2+\|\nabla\bu(t)\|_{\H}\|\boldsymbol{\eta}(t)\|_{\mathbb{L}^4}^2+\|\W(t)\|_{\H}\|\boldsymbol{\eta}(t)\|_{\H}\nonumber\\&\quad+2\|\U(t)\|_{\H}\|\boldsymbol{\eta}(t)\|_{\H}^2\nonumber\\&\leq \frac{\nu}{2} \|\nabla\boldsymbol{\eta}(t)\|_{\H}^2+\left(5\nu+\frac{\sqrt{2}}{\nu}\right)\|\nabla\bu(t)\|_{\H}^2\|\boldsymbol{\eta}(t)\|_{\H}^2+\frac{1}{2}\|\W(t)\|_{\H}^2+\frac{1}{2}\|\boldsymbol{\eta}(t)\|_{\H}^2\nonumber\\&\quad+2\|\U(t)\|_{\H}\|\boldsymbol{\eta}(t)\|_{\H}^2,
	\end{align}
	for a.e. $t\in[0,T]$. Integrating the above inequality from $0$ to $t$ and then an applying Gr\"onwall's inequality, we find 
	\begin{align}
		&	\|\boldsymbol{\eta}(t)\|_{\H}^2+\nu\int_0^t\|\nabla\boldsymbol{\eta}(s)\|_{\H}^2\d s
		\nonumber\\&\leq \left\{\|\boldsymbol{\eta}_0\|_{\H}^2+\int_0^T\|\W(t)\|_{\H}^2\d t\right\}\exp\left\{T+2\left(5\nu+\frac{\sqrt{2}}{\nu}\right)\int_0^T\|\nabla\bu(t)\|_{\H}^2\d t+4\int_0^T\|\U(t)\|_{\H}\d t\right\},
	\end{align}
	for all $t\in[0,T]$. Therefore, we have $\boldsymbol{\eta}\in\mathrm{L}^{\infty}(0,T;\H)\cap\mathrm{L}^2(0,T;\V)$. Moreover, for all $\boldsymbol{\varphi}\in\mathrm{L}^2(0,T;\V)$, we have 
{\small{	\begin{align*}
		&\left|\int_0^T\left\langle\frac{\d\boldsymbol{\eta}(t)}{\d t},\boldsymbol{\varphi}(t)\right\rangle\d t\right|\nonumber\\&\leq\nu\left|\int_0^T(\nabla\boldsymbol{\eta}(t),\boldsymbol{\varphi}(t))\d t\right|+\nu\left|\int_0^T\|\nabla\bu(t)\|_{\H}^2(\boldsymbol{\eta}(t),\boldsymbol{\varphi}(t))\d t\right|\nonumber\\&\quad+2\nu\left|\int_0^T(\nabla\bu(t),\nabla\boldsymbol{\eta}(t))(\bu(t),\boldsymbol{\varphi}(t))\d t\right|+\left|\int_0^T\langle\B'(\bu(t))\boldsymbol{\eta}(t),\boldsymbol{\varphi}(t)\rangle\d t\right|+\left|\int_0^T(\W(t),\boldsymbol{\varphi}(t))\d t\right|\nonumber\\&\quad+\left|\int_0^T(\U(t),\bu(t))(\boldsymbol{\eta}(t),\boldsymbol{\varphi}(t))\d t\right|+\left|\int_0^T(\U(t),\boldsymbol{\eta}(t))(\bu(t),\boldsymbol{\varphi}(t))\d t\right|\nonumber\\&\leq \Bigg\{\nu\left(\int_0^T\|\nabla\boldsymbol{\eta}(t)\|_{\H}^2\d t\right)^{1/2}+\frac{\nu T^{1/2}}{\lambda_1^{1/2}}\sup_{t\in[0,T]}\|\nabla\bu(t)\|_{\H}^2\sup_{t\in[0,T]}\|\boldsymbol{\eta}(t)\|_{\H}\nonumber\\&\quad+\frac{2\nu}{\lambda_1^{1/2}}\sup_{t\in[0,T]}\|\nabla\bu(t)\|_{\H}\left(\int_0^T\|\nabla\boldsymbol{\eta}(t)\|_{\H}^2\d t\right)^{1/2}+2\left(\int_0^T\|\bu(t)\|_{\mathbb{L}^4}^4\d t\right)^{1/4}\left(\int_0^T\|\boldsymbol{\eta}(t)\|_{\mathbb{L}^4}^4\d t\right)^{1/4}\nonumber\\&\quad+\frac{1}{\lambda_1^{1/2}}\left(\int_0^T\|\W(t)\|_{\H}^2\d t\right)^{1/2}+\frac{2}{\lambda_1^{1/2}}\sup_{t\in[0,T]}\|\boldsymbol{\eta}(t)\|_{\H}\left(\int_0^T\|\U(t)\|_{\H}^2\d t\right)^{1/2}\Bigg\}\left(\int_0^T\|\nabla\boldsymbol{\varphi}(t)\|_{\H}^2\d t\right)^{1/2},
	\end{align*}}}
	where $\lambda_1$ is the first eigenvalue of the Stokes operator. Therefore, we deduce $\frac{\d\boldsymbol{\eta}}{\d t}\in\mathrm{L}^2(0,T;\V^{\prime})$ and an application of the Lions-Magenes lemma yields $\boldsymbol{\eta}\in\mathrm{C}([0,T];\H)$.

Based on the above discussion, we now state the existence and uniqueness result for weak solutions to problem \eqref{4}.
	\begin{theorem}\label{linear}
		Let  $\boldsymbol{\eta}_0\in\mathbb{H}$ and $\W,\U\in\mathrm{L}^2(0,T;\H)$ be given.  Then,
		there exists \emph{a unique weak solution} to problem \eqref{4} satisfying
		$$\boldsymbol{\eta}\in\mathrm{L}^{\infty}(0,T;\H)\cap\mathrm{L}^2(0,T;\V), $$ with $\frac{\d\boldsymbol{\eta}}{\d t}\in\mathrm{L}^{2}(0,T;\mathbb{V}'),$ (so that $\boldsymbol{\eta}\in\C([0,T];\H)$),  the energy estimate 
		\begin{align*}
			\sup_{t\in[0,T]}	\|\boldsymbol{\eta}(t)\|_{\H}^2+\nu\int_0^T\|\boldsymbol{\eta}(t)\|_{\V}^2\d t
			\leq C_{\nu,T}\left\{\|\boldsymbol{\eta}_0\|_{\H}^2+\int_0^T\|\W(t)\|_{\H}^2\d t\right\},
			\end{align*}
		and the energy equality
		\begin{align*}
			&	\|\boldsymbol{\eta}(t)\|_{\H}^2+2\nu\int_0^t\|\boldsymbol{\eta}(s)\|_{\V}^2\d s\nonumber\\&=\|\boldsymbol{\eta}_0\|_{\H}^2+2\nu\int_0^t\|\nabla\bu(s)\|_{\H}^2\|\boldsymbol{\eta}(s)\|_{\H}^2\d s+4\nu\int_0^t(\nabla\bu(s),\nabla\boldsymbol{\eta}(s))(\bu(s),\boldsymbol{\eta}(s))\d s\nonumber\\&\quad-2\int_0^t\langle\B(\boldsymbol{\eta}(s),\bu(s)),\boldsymbol{\eta}(s)\rangle\d s+2\int_0^t(\W(s),\boldsymbol{\eta}(s))\d s\nonumber\\&\quad-2\int_0^t([(\U(s),\bu(s))\boldsymbol{\eta}(s)+(\U(s),\boldsymbol{\eta}(s))\bu(s)],\boldsymbol{\eta}(s))\d s,
		\end{align*}
		for all $t\in[0,T]$.
	\end{theorem}
	
		\begin{remark}
		In the linearized problem under our consideration (see \eqref{4a} below), we take $\W$ as $\U-(\U,\bu_{\U^*})\bu_{\U^*}$ for some control $\U\in\mathrm{L}^2(0,T;\H)$. Since $\bu_{\U^*}(t)\in\mathcal{M}$ for all $t\in[0,T]$, it is immediate that 
		\begin{align*}
			\int_0^T\|\W(t)\|_{\H}^2\d t&=\int_0^T\left\|\U(t)-(\U(t),\bu_{\U^*}(t))\bu_{\U^*}(t)\right\|_{\H}^2\d t\nonumber\\&=\int_0^T\|\U(t)\|_{\H}^2\d t-\int_0^T|(\U(t),\bu_{\U^*}(t))|^2\d t\leq 2\int_0^T\|\U(t)\|_{\H}^2\d t<\infty,
		\end{align*}
		which ensures the existence of a weak solution for the problem \eqref{4a}. 
	\end{remark}

	\subsection{The adjoint system} To obtain the first-order necessary optimality conditions, we derive the adjoint system corresponding to \eqref{eqn-constrained-NSE-N}.	Recall that the optimal control is expressed in terms of the adjoint variable. Let the adjoint variable $\p(\cdot)$ be defined as the solution to the following system:
	\begin{eqnarray}\label{adj}
		\left\{
		\begin{aligned}
			-\frac{\d\p(t)}{\d t}&=-\nu\A\p(t)+\nu\|\nabla\bu(t)\|_{\H}^2\p(t)+2\nu(\bu(t),\p(t))\A\bu(t)-(\B'(\bu(t)))^*\p (t)\\&\quad-\left[(\U(t),\bu(t))\p(t)+(\bu(t),\p(t))\U(t)\right]+(\bu(t)-\bu_d(t))+\A\bu(t),\ \text{ in }\ \V', \\ \p(T)&=\bu(T)-\bu_f,
		\end{aligned}
		\right.
	\end{eqnarray}
	for a.e. $t\in[0,T]$,  where $\langle(\B'(\bu))^*\p,\boldsymbol{q}\rangle=\langle\p,\B'(\bu)\boldsymbol{q}\rangle=\langle\B(\bu,\boldsymbol{q}),\p\rangle+\langle\B(\boldsymbol{q},\bu),\p\rangle$, for all $\bu,\p,\boldsymbol{q}\in\V$. Let us now obtain a-priori energy estimates satisfied by $\p$. Taking the inner product with $\p(\cdot)$ to the first equation in \eqref{adj}, using the fact that $\langle\B(\bu,\p),\p\rangle=0$ and  $\bu(t)\in\mathcal{M}$ for all $t\in[0,T]$, then applying H\"older's, Ladyzhenskaya's and Young's inequalities,  we find 
	\begin{align}\label{37}
		-&\frac{1}{2}\frac{\d}{\d t}\|\p(t)\|_{\H}^2+\nu\|\nabla\p(t)\|_{\H}^2\nonumber\\&=\nu\|\nabla\bu(t)\|_{\H}^2\|\p(t)\|_{\H}^2+2\nu(\bu(t),\p(t))(\nabla\bu(t),\nabla\p(t))-\langle\B(\p(t),\bu(t)),\p(t)\rangle\nonumber\\&\quad-\left[(\U(t),\bu(t))\|\p(t)\|_{\H}^2+(\bu(t),\p(t))(\U(t),\p(t))\right]+(\bu(t)-\bu_d(t),\p(t))+(\nabla\bu(t),\nabla\p(t))\nonumber\\&\leq \frac{3\nu}{4}\|\nabla\p(t)\|_{\H}^2+\left(5\nu+\frac{\sqrt{2}}{\nu}\right)\|\nabla\bu(t)\|_{\H}^2\|\p(t)\|_{\H}^2\nonumber\\&\quad+2\|\U(t)\|_{\H}\|\p(t)\|_{\H}^2+\frac{1}{2}\|\bu(t)-\bu_d\|_{\H}^2+\frac{1}{2}\|\p(t)\|_{\H}^2+\frac{1}{\nu}\|\nabla\bu(t)\|_{\H}^2,
	\end{align}
	for a.e. $t\in[0,T]$. Integrating the above inequality from $t$ to $T$, we obtain for all $t\in[0,T]$ 
	\begin{align}\label{39}
		&\|\p(t)\|_{\H}^2+\frac{\nu}{2}\int_t^T\|\nabla\p(s)\|_{\H}^2\d s\nonumber\\&\leq\|\p(T)\|_{\H}^2+\int_t^T\|\bu(s)-\bu_d\|_{\H}^2\d s+\frac{2}{\nu}\int_t^T\|\nabla\bu(s)\|_{\H}^2\d s \nonumber\\&\quad+2\left(5\nu+\frac{\sqrt{2}}{\nu}\right)\int_t^T\|\nabla\bu(s)\|_{\H}^2\|\p(s)\|_{\H}^2\d s+\int_t^T\left(1+4\|\U(s)\|_{\H}\right)\|\p(s)\|_{\H}^2\d s . 
	\end{align}
	An application of Gr\"onwall's inequality in \eqref{39} yields
	\begin{align}\label{40}
		&\sup_{t\in[0,T]}\|\p(t)\|_{\H}^2+\frac{\nu}{2}\int_0^T\|\nabla\p(t)\|_{\H}^2\d t\nonumber\\&\leq \left(\|\bu(T)-\bu_f\|_{\H}^2+\int_0^T\|\bu(t)-\bu_d\|_{\H}^2\d t+\frac{2}{\nu}\int_0^T\|\nabla\bu(t)\|_{\H}^2\d t\right)\nonumber\\&\quad\times\exp\left\{T+4\int_0^T\|\U(t)\|_{\H}\d t+2\left(5\nu+\frac{\sqrt{2}}{\nu}\right)\int_0^T\|\nabla\bu(t)\|_{\H}^2\d t\right\},
	\end{align}
	for all $t\in[0,T]$. Hence $\p\in\mathrm{L}^{\infty}(0,T;\H)\cap\mathrm{L}^2(0,T;\V)$. 
	
	For all $\boldsymbol{\varphi}\in\mathrm{L}^2(0,T;\V)$, we compute 
	\begin{align}
		&	\left|\int_0^T\langle\partial_t\p(t),\boldsymbol{\varphi}(t)\rangle\d t\right|\nonumber\\&\leq
		\nu\left|\int_0^T(\nabla\p(t),\nabla\boldsymbol{\varphi}(t))\d t\right|+\nu\left|\int_0^T\|\nabla\bu(t)\|_{\H}^2(\p(t),\boldsymbol{\varphi}(t))\d t\right|\nonumber\\&\quad+2\nu\left|\int_0^T(\bu(t),\p(t))(\nabla\bu(t),\nabla\boldsymbol{\varphi}(t))\d t\right|+\left|\int_0^T\langle(\B'(\bu(t)))^*\p (t),\boldsymbol{\varphi}(t)\rangle\d t\right|\nonumber\\&\quad+\left|\int_0^T(\left[(\U(t),\bu(t))\p(t)+(\bu(t),\p(t))\U(t)\right],\boldsymbol{\varphi}(t))\d t\right|\nonumber\\&\quad+\left|\int_0^T(\bu(t)-\bu_d(t),\boldsymbol{\varphi}(t))\d t\right|+\left|\int_0^T(\nabla\bu(t),\nabla\boldsymbol{\varphi}(t))\d t\right|\nonumber\\&\leq\Bigg\{\nu\left(\int_0^T\|\nabla\p(t)\|_{\H}^2\d t\right)^{1/2}+\nu\lambda_1\sup_{t\in[0,T]}\|\p(t)\|_{\H}\int_0^T\|\nabla\bu(t)\|_{\H}^2\d t\nonumber\\&\quad+2\nu\sup_{t\in[0,T]}\|\p(t)\|_{\H}\left(\int_0^T\|\nabla\bu(t)\|_{\H}^2\d t\right)^{1/2}+\left(\int_0^T\|\bu(t)\|_{\mathbb{L}^4}^4\d t\right)^{1/4}\left(\int_0^T\|\p(t)\|_{\mathbb{L}^4}^4\d t\right)^{1/4}\nonumber\\&\quad+\left(\frac{2}{\lambda_1}\right)^{1/4}\sup_{t\in[0,T]}\|\nabla\bu(t)\|_{\H}\left(\int_0^T\|\p(t)\|_{\mathbb{L}^4}^2\d t\right)^{1/2}+\frac{2}{\lambda_1^{1/2}}\sup_{t\in[0,T]}\|\p(t)\|_{\H}\left(\int_0^T\|\U(t)\|_{\H}^2\d t\right)^{1/2}
		\nonumber\\&\quad+\frac{T^{1/2}}{\lambda_1^{1/2}}\sup_{t\in[0,T]}\|\bu(t)-\bu_d(t)\|_{\H}+\left(\int_0^T\|\nabla\bu(t)\|_{\H}^2\d t\right)^{1/2}\Bigg\}\left(\int_0^T\|\nabla\boldsymbol{\varphi}(t)\|_{\H}^2\d t\right)^{1/2}.
	\end{align}
	Therefore, $\frac{\d\p}{\d t}\in\mathrm{L}^2(0,T;\V^{\prime})$, so that an application of the Lions-Magenes lemma provides $\p\in\mathrm{C}([0,T];\H)$.

	The above discussion yields the following existence and uniqueness result for weak solutions of the adjoint problem \eqref{adj}.
	\begin{theorem}\label{thm3.4}
		For $\U\in\mathrm{L}^2(0,T;\H)$,
		there exists a \emph{unique weak solution} to problem \eqref{adj} satisfying
		$$\p\in\mathrm{L}^{\infty}(0,T;\H)\cap\mathrm{L}^2(0,T;\V), $$
		with $\frac{\d\p}{\d t}\in\mathrm{L}^{2}(0,T;\mathbb{V}')$ (so that $\p\in\C([0,T];\H)$), the energy estimate 
		\begin{align}\label{adjoint-energy-est}
				&\sup_{t\in[0,T]}\|\p(t)\|_{\H}^2+\frac{\nu}{2}\int_0^T\|\nabla\p(t)\|_{\H}^2\d t\nonumber\\&\leq C_{\nu,T}\left(\|\bu(T)-\bu_f\|_{\H}^2+\int_0^T\|\bu(t)-\bu_d\|_{\H}^2\d t+\int_0^T\|\nabla\bu(t)\|_{\H}^2\d t\right),
		\end{align}
			and the energy equality 
		\begin{align*}
			&	\|\p(t)\|_{\H}^2+2\nu\int_t^T\|\p(s)\|_{\V}^2\d s\nonumber\\&=\|\p_T\|_{\H}^2+2\nu\int_t^T\|\nabla\bu(s)\|_{\H}^2\|\p(s)\|_{\H}^2\d s+4\nu\int_t^T(\bu(s),\p(s))(\nabla\bu(s),\nabla\p(s))\d s\nonumber\\&\quad-2\int_t^T\langle\B(\p(s),\bu(s)),\p(s)\rangle\d s-2\int_t^T(\left[(\U(s),\bu(s))\p(s)+(\bu(s),\p(s))\U(s)\right],\p(s))\d s
			\nonumber\\&\quad+2\int_t^T(\bu(s)-\bu_d(s),\p(s))\d s+2\int_t^T(\nabla\bu(s),\nabla\p(s))\d s,
		\end{align*}
		for all $t\in[0,T]$, where $\p_T=\bu(T)-\bu_f$. 
	\end{theorem}

	\section{First Order Necessary Conditions of Optimality}\label{sec4}
In this section, we derive the first-order necessary condition for optimality associated with the control problem \eqref{control problem}. The optimal control will be expressed in terms of the corresponding adjoint variable. We start by formally deriving \emph{Pontryagin’s maximum principle} through the \emph{Hamiltonian framework}. For this purpose, we introduce the \emph{Lagrangian} defined as follows:
$$\mathscr{L}(\bu,\U) :=\frac{1}{2}\|\bu-\bu_d\|_{\H}^2 +\frac{1}{2}\|\nabla\bu\|_{\H}^2+\frac{1}{2} \|\U\|_{\H}^2, $$ so that $\mathscr{J}(\bu,\U)=\int_0^T\mathscr{L}(\bu(t),\U(t))\d t+\frac{1}{2}\|\bu(T)-\bu^f\|_{\H}^2$.
	We define the corresponding \emph{Hamiltonian} by
	$$\mathscr{H}(\bu,\U,\p):= \mathscr{L}(\bu,\U) + \langle\p, \mathcal{N}(\bu,\U) \rangle,$$ where $$\mathcal{N}(\bu,\U)=-\nu\A\bu+\nu\|\nabla\bu\|_{\H}^2\bu-\B(\bu,\bu)+\U-(\U,\bu)\bu$$ and the adjoint variable $\p(\cdot)$ satisfies the system \eqref{adjp} given below. Since $(\bu^*, \U^*)$ is an optimal pair for problem \eqref{control problem}, Pontryagin’s minimum principle yields
	\begin{eqnarray}
		\mathscr{H}(\bu^*(t),\U^*(t),\p(t)) \leq \mathscr{H}(\bu^*(t),\W,\p(t)),
	\end{eqnarray}
	for all $\W \in \mathbb{U}$ and a.e. $t\in[0,T]$. That is, the following minimum principle is satisfied by an optimal pair $(\bu^*,\U^*)\in\mathscr{A}_{\mathrm{ad}}$ obtained in Theorem \ref{thm-exis-optimal}:
	\begin{align}\label{pm}
		&\frac{1}{2} \|\U^*(t)\|_{\H}^2+(\p(t),\U^*(t)-(\U^*(t),\bu^*(t))\bu^*(t)) \nonumber\\&\leq \frac{1}{2} \|\W\|_{\H}^2+( \p(t),\W-(\W,\bu^*(t))\bu^*(t)),
	\end{align}
	for all  $\boldsymbol{W}\in \mathbb{U},$ and a.e. $t\in[0,T]$. The above relation implies
	\begin{align*}
		\frac{1}{2}\left(\|\W\|_{\H}^2-\|\U^*(t)\|_{\H}^2\right)+\left(\p(t)-(\p(t),\bu^*(t))\bu^*(t),\W-\U^*(t)\right)\geq 0.
	\end{align*}
	Since  $\U^*,\W\in\mathscr{U}_{\mathrm{ad}}=\mathrm{L}^2(0,T;\H)$, we deduce  from \eqref{pm} that 
	\begin{align*} \mbox{$-\p(t)+(\p(t),\bu^*(t))\bu^*(t) \in \partial \frac{1}{2}\|\U^*(t)\|_{\H}^2$, for a.e. $t\in[0,T],$}
		\end{align*} where $\partial h(\cdot)$ denotes the subdifferential of $h(\cdot)$. Since $\frac{1}{2}\|\U^*\|_{\H}^2$ is Fr\'echet differentiable, so that the subdifferential contains a single point, and hence we deduce 
	\begin{align}
		\U^*(t)=-\p(t)+(\p(t),\bu^*(t))\bu^*(t)=-\pi_{\bu^*(t)}(\p(t)), \ \text{ for a.e. } \ t\in[0,T]\ \text{ in }\ \H. 
	\end{align}
	Note that the optimal control $\U^*$ is orthogonal to $\bu^*$ in $\H$ ($\U^*\perp\bu^*$ in $\H$), since 
	\begin{align*}
		(\U^*(t),\bu^*(t))=-(\p(t),\bu^*(t))+(\p(t),\bu^*(t))\|\bu^*(t)\|_{\H}^2=0,\ \text{ for a.e. }\ t\in[0,T],
	\end{align*}
	where we have used the fact that $\bu^*(t)\in\mathcal{M}$ for all $t\in[0,T]$. Since $\p\in\mathrm{C}([0,T];\H)$ and $\bu^*\in\mathrm{C}([0,T];\V)$,   we infer  $\U^*(t)\in T_{\bu^*(t)}\mathcal{M},$ for all $t\in[0,T].$ 	Recall $\mathcal M=\{\bu\in\H:\|\bu\|_{\H}=1\}$ and 
	$T_{\bu}\mathcal M=\{\v\in\H:(\v,\bu)=0\}$. The tangent bundle \(T\mathcal M=\bigsqcup_{\bu\in\mathcal M}T_{\bu}\mathcal M\)
	is a fixed manifold (the disjoint union of the tangent spaces). A trajectory
	\(\bu(\cdot)\) in \(\mathcal M\) determines at each time \(t\) the
	instantaneous tangent space \(T_{\bu(t)}\mathcal M\); these spaces vary
	with time only because the base point \(\bu(t)\) varies. An optimal control \(\U^*(\cdot)\) is then a time-dependent section
	along \(\bu^*(\cdot)\), that is, for every \(t\) we have
	\(\U^*(t)\in T_{\bu^*(t)}\mathcal M\). Thus the map
$
	t\mapsto\big(\bu^*(t),\U^*(t)\big)
$
	is a curve in the fixed manifold \(T\mathcal M\). 
	
As Pontryagin’s maximum principle provides the first-order necessary conditions for optimality, we now summarize the resulting conclusions in the theorem below. The proof of this result is based on analytical methods comparable to those developed in \cite{AF+TR-90,MTM-2022,SSS-1998}, among others. The principal result of this section is formulated as follows:
	\begin{theorem}\label{main}
		Let $(\bu^*,\U^*)\in\mathscr{A}_{\mathrm{ad}}$ be an optimal solution to problem \eqref{control problem} obtained in  Theorem \ref{thm-exis-optimal}. Then, there exists \emph{a unique weak solution} $\p\in\mathrm{W}^{1,2}(0,T;\V,\V^{\prime})$ of the adjoint system:
		\begin{eqnarray}\label{adjp}
			\left\{
			\begin{aligned}
				-\frac{\d\p(t)}{\d t}&=-\nu\A\p(t)+\nu\|\nabla\bu^*(t)\|_{\H}^2\p(t)+2\nu(\bu^*(t),\p(t))\A\bu^*(t)-(\B'(\bu^*(t)))^*\p (t)\\&\quad-\left[(\U^*(t),\bu^*(t))\p(t)+\U^*(t)(\bu^*(t),\p(t))\right]+(\bu^*(t)-\bu_d(t))+\A\bu^*(t),\ \text{ in }\ \V', \\ \p(T)&=\bu^*(T)-\bu_f\in\H,
			\end{aligned}
			\right.
		\end{eqnarray}
		for	a.e. $t\in[0,T]$,  	such that
		\begin{equation}\label{3.41}
			\U^*(t)=-\pi_{\bu^*(t)}(\p(t))=-\p(t)+(\p(t),\bu^*(t))\bu^*(t), \ \text{ for all } \ t\in[0,T]\ \text{ in }\ \H. 
		\end{equation}
		Moreover $\U^*(t)\perp\bu^*(t)$  in $\H$ for all $t\in[0,T]$, that is, $\U^*(t)\in T_{\bu^*(t)}\mathcal{M}$ for all  $t\in[0,T].$ The pair $(\bu^*(\cdot),\U^*(\cdot))$ is an element
		of the tangent bundle $T\mathcal M.$
	\end{theorem}
	
	\begin{remark}
		Since $\U^*(t)\perp\bu^*(t)$  in $\H$ for all $t\in[0,T]$, under an optimal control $(\bu^*,\U^*)$ for problem \eqref{3.41}, the constrained PDE reduces to 
			\begin{equation}\label{eqn-constrained-NSE-W}
			\left\{
			\begin{aligned}
				\frac{\d\bu^*(t)}{\d t}&=-\nu\A\bu^*(t)+\nu\|\nabla\bu^*(t)\|_{\H}^2\bu^*(t)-\B(\bu^*(t),\bu^*(t))+\U^*(t),\ \text{ in }\ \H, \\
				\bu^*(0)&=\bu_0,
			\end{aligned}
			\right. 
		\end{equation} 
	and 	the adjoint system \eqref{adjp} reduces to 
		\begin{eqnarray}\label{adjp-new}
			\left\{
			\begin{aligned}
				-\frac{\d\p(t)}{\d t}&=-\nu\A\p(t)+\nu\|\nabla\bu^*(t)\|_{\H}^2\p(t)+2\nu(\bu^*(t),\p(t))\A\bu^*(t)-(\B'(\bu^*(t)))^*\p (t)\\&\quad-\left[\p(t)-(\p(t),\bu^*(t))\bu^*(t)\right](\bu^*(t),\p(t))+(\bu^*(t)-\bu_d(t))+\A\bu^*(t),\ \text{ in }\ \V', \\ \p(T)&=\bu^*(T)-\bu_f\in\H,
			\end{aligned}
			\right.
		\end{eqnarray}
		for	a.e. $t\in[0,T]$. 
	\end{remark}
Before proving Theorem \ref{main}, we present several auxiliary results that will be useful in establishing the first-order necessary conditions for optimality. 

	\begin{lemma}\label{lem3.7}
		Let $(\bu^*, \U^*) \in \mathscr{A}_{\mathrm{ad}}$ denote the optimal pair for the control problem \eqref{control problem} established in Theorem \ref{thm-exis-optimal}.
		For any sufficiently small $\tau > 0$ and any $\U \in \mathscr{U}_{\mathrm{ad}}$, let $\bu_{\U^* + \tau \U}$ represent the unique strong solution of system \eqref{eqn-constrained-NSE-N} corresponding to the control $\U^* + \tau \U$.
		Then the following result holds:
		\begin{align}\label{325a}
			&\sup_{t\in[0,T]}\|\bu_{\U^*+\tau\U}(t)-\bu_{\U^*}(t)\|_{\V}^2+\nu\int_0^T\|\A\left(\bu_{\U^*+\tau\U}(t)-\bu_{\U^*}(t)\right)\|_{\H}^2\d t\nonumber\\&\leq  \tau^2 K_{\nu,T}\left(\int_0^T\|\U(t)\|_{\H}^2\d t\right),
		\end{align}
		where 
		\begin{align}
			K_{\nu,T}&=C_{\nu,T}\bigg[1+\left(\nu\|\bu_0\|_{\V}^2+\int_0^T\|\U^*(t)\|_{\H}^2\d t+\int_0^T\|\U(t)\|_{\H}^2\d t\right)^2\bigg]\nonumber\\&\quad\times
			\exp\bigg\{
			C_{\nu,T}	\bigg[1+\left(\nu\|\nabla\bu_0\|_{\H}^2+\int_0^T\|\U^*(t)\|_{\H}^2\d t+\int_0^T\|\U(t)\|_{\H}^2\d t\right)^2\bigg]\bigg\}.
		\end{align}
	\end{lemma}
	\begin{proof}
		Let $(\bu^*, \U^*) \in \mathscr{A}_{\mathrm{ad}}$ be the optimal pair associated with the control problem \eqref{control problem}, as established in Theorem \ref{thm-exis-optimal}.
		For every $\U \in \mathscr{U}_{\mathrm{ad}}$ and sufficiently small $\tau > 0$, denote by $\bu_{\U^* + \tau \U}$, the unique strong solution of system \eqref{eqn-constrained-NSE-N} corresponding to the perturbed control $\U^* + \tau \U$. Under these assumptions, the following estimates are obtained:
		\begin{align}\label{eqn-difference-9}
			&	\nu\|\bu_{\U^*}(t)\|_{\V}^2+\int_0^T\left\|\frac{\d\bu_{\U^*}(t)}{\d t}\right\|_{\H}^2\d t+\nu^2\int_0^T\|\A\bu_{\U^*}(t)\|_{\H}^2\d t\nonumber\\&\leq C_{\nu,T}\left[1+\left(\nu\|\nabla\bu_0\|_{\H}^2+\int_0^T\|\U^*(t)\|_{\H}^2\d t\right)^2\right]
		\end{align}
		and 
		\begin{align}\label{eqn-difference-11}
			&	\nu\|\bu_{\U^*+\tau \U}(t)\|_{\V}^2+\int_0^T\left\|\frac{\d\bu_{\U^*+\tau \U}(t)}{\d t}\right\|_{\H}^2\d t+\nu^2\int_0^T\|\A\bu_{\U^*+\tau \U}(t)\|_{\H}^2\d t\nonumber\\&\leq C_{\nu,T}\left[1+\left(\nu\|\nabla\bu_0\|_{\H}^2+\int_0^T\|\U^*(t)\|_{\H}^2\d t+\int_0^T\|\U(t)\|_{\H}^2\d t\right)^2\right],
		\end{align}
		for any $\tau\in(0,1]$. 
		We know that $\widetilde{\bu}=\bu_{\U^*+\tau \U}-\bu_{\U^*}$  satisfies the following system for $t\in(0,T)$:
		\begin{equation}\label{11}
			\left\{
			\begin{aligned}
				\frac{\d \widetilde{\boldsymbol{u}}(t)}{\d t}&=-\nu\A \widetilde{\boldsymbol{u}}(t)+\nu\|\nabla\bu_{\U^*}(t)\|_{\H}^2\widetilde{\boldsymbol{u}}(t)+\nu\left(\|\nabla\bu_{\U^*+\tau \U}(t)\|_{\H}^2-\|\nabla\bu_{\U^*}(t)\|_{\H}^2\right)\bu_{\U^*+\tau\U}(t)\\&\quad-[\B(\bu_{\U^*+\tau \U}(t))-\B(\bu_{\U^*}(t))]+\tau\U(t)\\&\quad-\left[(\U^*(t)+\tau \U(t),\bu_{\U^*+\tau \U}(t))\bu_{\U^*+\tau \U}(t)-(\U^*(t),\bu_{\U^*}(t))\bu_{\U^*}(t)\right],\ \text{ in }\ \V^{\prime},\\
				\widetilde{\bu}(0)&=\boldsymbol{0}.
			\end{aligned}
			\right.
		\end{equation}
		Taking the inner product with $\widetilde{\bu}(\cdot)$ to the first equation in \eqref{11}, we obtain for a.e. $t\in[0,T]$ 
		\begin{align}\label{eqn-difference}
			&	\frac{1}{2}\frac{\d }{\d t}\|\widetilde{\bu}(t)\|_{\H}^2+\nu\|\nabla\widetilde{\bu}(t)\|_{\H}^2\nonumber\\&=\nu\|\nabla\bu_{\U^*}(t)\|_{\H}^2\|\widetilde{\boldsymbol{u}}(t)\|_{\H}^2+\nu\left(\|\nabla\bu_{\U^*+\tau \U}(t)\|_{\H}^2-\|\nabla\bu_{\U^*}(t)\|_{\H}^2\right)(\bu_{\U^*+\tau \U}(t),\widetilde{\bu}(t))\nonumber\\&\quad-\langle\B(\bu_{\U^*+\tau \U}(t))-\B(\bu_{\U^*}(t)),\widetilde{\bu}(t)\rangle+\tau(\U(t),\widetilde{\bu}(t))\nonumber\\&\quad-\left(\left[(\U^*(t)+\tau \U(t),\bu_{\U^*+\tau \U}(t))\bu_{\U^*+\tau \U}(t)-(\U^*(t),\bu_{\U^*}(t))\bu_{\U^*}(t)\right],\widetilde{\bu}(t)\right)\nonumber\\&=\nu\|\nabla\bu_{\U^*+\tau \U}(t)\|_{\H}^2\|\widetilde{\boldsymbol{u}}(t)\|_{\H}^2+\nu\left(\|\nabla\bu_{\U^*}(t)\|_{\H}^2-\|\nabla\bu_{\U^*}(t)\|_{\H}^2\right)(\bu_{\U^*+\tau \U}(t),\widetilde{\bu}(t))\nonumber\\&\quad-(\B(\widetilde{\bu}(t),\bu_{\U^*}(t)),\widetilde{\bu}(t))+\tau(\U(t),\widetilde{\bu}(t))-\tau(\U(t),\bu_{\U^*+\tau \U}(t))(\bu_{\U^*+\tau \U}(t),\widetilde{\bu}(t))\nonumber\\&\quad-(\U^*(t),\widetilde{\bu}(t))(\bu_{\U^*+\tau \U}(t),\widetilde{\bu}(t))-(\U^*(t),\bu_{\U^*}(t))\|\widetilde{\bu}(t)\|_{\H}^2.
		\end{align}
		Let us estimate $\nu\left(\|\nabla\bu_{\U^*+\tau \U}\|_{\H}^2-\|\nabla\bu_{\U^*+\tau \U}\|_{\H}^2\right)(\bu_{\U^*},\widetilde{\bu})$  	using  the fact that $\bu_{\U^*+\tau \U}\in\mathcal{M}$ as 
		\begin{align}\label{eqn-difference-1}
			&\nu	\left(\|\nabla\bu_{\U^*+\tau \U}\|_{\H}^2-\|\nabla\bu_{\U^*}\|_{\H}^2\right)(\bu_{\U^*+\tau \U},\widetilde{\bu})\nonumber\\&\leq \nu\left(\|\nabla\bu_{\U^*+\tau \U}\|_{\H}+\|\nabla\bu_{\U^*}\|_{\H}\right)\|\nabla(\bu_{\U^*+\tau \U}-\bu_{\U^*})\|_{\H}\|\bu_{\U^*+\tau \U}\|_{\H}\|\widetilde{\bu}\|_{\H}\nonumber\\&\leq \frac{\nu}{4}\|\nabla\widetilde{\bu}\|_{\H}^2+\nu\left(\|\nabla\bu_{\U^*+\tau \U}\|_{\H}+\|\nabla\bu_{\U^*}\|_{\H}\right)^2\|\widetilde{\bu}\|_{\H}^2.
		\end{align}
	We estimate $-(\B(\widetilde{\bu},\bu_{\U^*}),\widetilde{\bu})$ using H\"older's, Ladyzhenskaya's and Young's inequalities as
		\begin{align}
			-(\B(\widetilde{\bu},\bu_{\U^*}),\widetilde{\bu})&\leq\|\nabla\bu_{\U^*}\|_{\H}\|\widetilde{\bu}\|_{\mathbb{L}^4}^2\leq 2^{1/2}\|\nabla\bu_{\U^*}\|_{\H}\|\widetilde{\bu}\|_{\H}\|\nabla\widetilde{\bu}\|_{\H}\nonumber\\&\leq\frac{\nu}{4}\|\nabla\widetilde{\bu}\|_{\H}^2+\frac{2}{\nu}\|\nabla\bu_{\U^*}\|_{\H}^2\|\widetilde{\bu}\|_{\H}^2. 
		\end{align}
		A simple application of the Cauchy-Schwarz and Young inequalities yield
		\begin{align}
			\tau(\U,\widetilde{\bu})\leq\tau\|\U\|_{\H}\|\widetilde{\bu}\|_{\H}\leq \frac{1}{2}\|\widetilde{\bu}\|_{\H}^2+\frac{\tau^2}{2}\|\U\|_{\H}^2.
		\end{align}
		Since $\bu_{\U^*+\tau \U}\in\mathcal{M}$, the expression $-\tau(\U,\bu_{\U^*+\tau \U})(\bu_{\U^*+\tau \U},\widetilde{\bu})$ can be bounded analogously to the previous case as
		\begin{align}
			-\tau(\U,\bu_{\U^*+\tau \U})(\bu_{\U^*+\tau \U},\widetilde{\bu})&\leq\tau\|\U\|_{\H}\|\widetilde{\bu}\|_{\H}\leq \frac{1}{2}\|\widetilde{\bu}\|_{\H}^2+ \frac{\tau^2}{2}\|\U\|_{\H}^2.
		\end{align}
		Once again using the fact that $\bu_{\U^*+\tau \U},\bu_{\U^*}\in\mathcal{M}$, we estimate the term $-(\U^*,\widetilde{\bu})(\bu_{\U^*+\tau \U},\widetilde{\bu})-(\U^*,\bu_{\U^*})\|\widetilde{\bu}\|_{\H}^2$ as 
		\begin{align}\label{eqn-difference-5}
			&-(\U^*,\widetilde{\bu})(\bu_{\U^*+\tau \U},\widetilde{\bu})-(\U^*,\bu_{\U^*})\|\widetilde{\bu}\|_{\H}^2\leq 2\|\U^*\|_{\H}\|\widetilde{\bu}\|_{\H}^2. 
		\end{align}
		Combining \eqref{eqn-difference-1}-\eqref{eqn-difference-5} and substituting it in \eqref{eqn-difference}, we deduce  for all $t\in[0,T]$ that 
	{\small{	\begin{align}
			&	\|\widetilde{\bu}(t)\|_{\H}^2+\nu\int_0^t\|\nabla\widetilde{\bu}(s)\|_{\H}^2\d s
			\nonumber\\&=\|\widetilde{\bu}(0)\|_{\H}^2+2\int_0^t\left[\left(2\nu\|\nabla\bu_{\U^*+\tau \U}(s)\|_{\H}^2+\left(3\nu+\frac{2}{\nu}\right)\|\nabla\bu_{\U^*}(s)\|_{\H}^2\right)+1+2\|\U^*(s)\|_{\H}\right]\|\widetilde{\bu}(s)\|_{\H}^2\d s\nonumber\\&\quad+\tau^2\int_0^t\|\U(s)\|_{\H}^2\d s+\tau^2\int_0^t\|\U^*(s)\|_{\H}^2\d s. 
		\end{align}}}
		Using the fact that $\widetilde{\bu}(0)=\boldsymbol{0}$ and applying Gr\"onwall's inequality, we conclude 
		\begin{align}\label{eqn-difference-8}
			&	\|\widetilde{\bu}(t)\|_{\H}^2+\nu\int_0^t\|\nabla\widetilde{\bu}(s)\|_{\H}^2\d s\nonumber\\&\leq  \left(\tau^2\int_0^T\|\U(t)\|_{\H}^2\d t+\tau^2\int_0^T\|\U^*(t)\|_{\H}^2\d t \right)\exp\bigg\{2\int_0^T\bigg[\left(2\nu\|\nabla\bu_{\U^*+\tau \U}(t)\|_{\H}^2\right.\nonumber\\&\left.\qquad+\left(3\nu+\frac{2}{\nu}\right)\|\nabla\bu_{\U^*}(t)\|_{\H}^2\right)+1+2\|\U^*(t)\|_{\H}\bigg]\d t\bigg\}\nonumber\\&\leq \left(2\tau^2\int_0^T\|\U(t)\|_{\H}^2\d t\right)\exp\bigg\{
			C_{\nu,T}	\bigg[1+\left(\nu\|\nabla\bu_0\|_{\H}^2+\int_0^T\|\U^*(t)\|_{\H}^2\d t+\int_0^T\|\U(t)\|_{\H}^2\d t\right)^2\bigg]\bigg\},
		\end{align}
		for all $t\in[0,T]$.

		Taking the inner product with $\A\widetilde{\bu}(\cdot)$ to the first equation in \eqref{11}, we infer for a.e. $t\in[0,T]$
		\begin{align}\label{eqn-difference-6}
			&	\frac{1}{2}\frac{\d }{\d t}\|\nabla\widetilde{\bu}(t)\|_{\H}^2+\nu\|\A\widetilde{\bu}(t)\|_{\H}^2\nonumber\\&=\nu\|\nabla\bu_{\U^*}(t)\|_{\H}^2\|\nabla\widetilde{\boldsymbol{u}}(t)\|_{\H}^2+\nu\left(\|\nabla\bu_{\U^*+\tau \U}(t)\|_{\H}^2-\|\nabla\bu_{\U^*}(t)\|_{\H}^2\right)(\nabla\bu_{\U^*+\tau \U}(t),\nabla\widetilde{\bu}(t))\nonumber\\&\quad-\langle\B(\bu_{\U^*+\tau \U}(t))-\B(\bu_{\U^*}(t)),\A\widetilde{\bu}(t)\rangle+\tau(\U(t),\A\widetilde{\bu}(t))\nonumber\\&\quad-\left(\left[(\U^*(t)+\tau \U(t),\bu_{\U^*+\tau \U}(t))\bu_{\U^*+\tau \U}(t)-(\U^*(t),\bu_{\U^*}(t))\bu_{\U^*}(t)\right],\A\widetilde{\bu}(t)\right)\nonumber\\&=\nu\|\nabla\bu_{\U^*}(t)\|_{\H}^2\|\nabla\widetilde{\boldsymbol{u}}(t)\|_{\H}^2+\nu\left(\|\nabla\bu_{\U^*+\tau \U}(t)\|_{\H}^2-\|\nabla\bu_{\U^*}(t)\|_{\H}^2\right)(\nabla\bu_{\U^*+\tau \U}(t),\nabla\widetilde{\bu}(t))\nonumber\\&\quad-(\B(\widetilde{\bu}(t),\bu_{\U^*}(t)),\A\widetilde{\bu}(t))+\tau(\U(t),\A\widetilde{\bu}(t))-\tau(\U(t),\bu_{\U^*+\tau \U}(t))(\bu_{\U^*+\tau \U}(t),\A\widetilde{\bu}(t))\nonumber\\&\quad-(\U^*(t),\widetilde{\bu}(t))(\bu_{\U^*+\tau \U}(t),\A\widetilde{\bu}(t))-(\U^*(t),\A\bu_{\U^*}(t))\|\nabla\widetilde{\bu}(t)\|_{\H}^2,
		\end{align}
	where we have used the fact that $(\B(\bu_{\U^*+\tau\U},\widetilde{\bu}),\A\widetilde{\bu})=0$.  
		Let us now estimate the term $\nu\left(\|\nabla\bu_{\U^*+\tau \U}\|_{\H}^2-\|\nabla\bu_{\U^*}\|_{\H}^2\right)(\nabla\bu_{\U^*+\tau \U},\nabla\widetilde{\bu})$  using H\"older's and Young's inequalities as 
		\begin{align}\label{eqn-difference-7}
			&	\nu\left(\|\nabla\bu_{\U^*+\tau \U}\|_{\H}^2-\|\nabla\bu_{\U^*}\|_{\H}^2\right)(\nabla\bu_{\U^*+\tau \U},\nabla\widetilde{\bu})\nonumber\\&\leq\nu\left(\|\nabla\bu_{\U^*+\tau \U}\|_{\H}+\|\nabla\bu_{\U^*}\|_{\H}\right)\|\nabla(\bu_{\U^*+\tau \U}-\bu_{\U^*})\|_{\H}\|\nabla\bu_{\U^*+\tau \U}\|_{\H}\|\nabla\widetilde{\bu}\|_{\H}\nonumber\\&=\nu\left(\|\nabla\bu_{\U^*+\tau \U}\|_{\H}+\|\nabla\bu_{\U^*}\|_{\H}\right)\|\nabla\bu_{\U^*+\tau \U}\|_{\H}\|\nabla\widetilde{\bu}\|_{\H}^2. 
		\end{align}
		We consider $-(\B(\widetilde{\bu},\bu_{\U^*}),\A\widetilde{\bu})$  and estimate it using Agmon's, Gagliardo-Nirenberg's and  Young's inequalities as 
		\begin{align}
			-(\B(\widetilde{\bu},\bu_{\U^*}),\A\widetilde{\bu})&\leq\|\widetilde{\bu}\|_{\mathbb{L}^{\infty}}\|\nabla\bu_{\U^*}\|_{\H}\|\A\widetilde{\bu}\|_{\H}\leq C\|\widetilde{\bu}\|_{\H}^{1/2}\|\nabla\bu_{\U^*}\|_{\H}\|\A\widetilde{\bu}\|_{\H}^{3/2}\nonumber\\&\leq\frac{\nu}{8}\|\A\widetilde{\bu}\|_{\H}^{2}+\frac{C}{\nu}\|\nabla\bu_{\U^*}\|_{\H}^4\|\widetilde{\bu}\|_{\H}^2. 
		\end{align}
		We estimate $\tau(\U,\A\widetilde{\bu})$ using the Cauchy-Schwarz  and Young inequalities as 
		\begin{align}
			\tau(\U,\A\widetilde{\bu})\leq\tau\|\U\|_{\H}\|\A\widetilde{\bu}\|_{\H}\leq\frac{\nu}{8}\|\A\widetilde{\bu}\|_{\H}^2+\frac{2\tau^2}{\nu}\|\U\|_{\H}^2. 
		\end{align}
		We use a similar calculation to estimate $-\tau(\U,\bu_{\U^*+\tau \U})(\bu_{\U^*+\tau \U},\A\widetilde{\bu})$ as 
		\begin{align}
			-\tau(\U,\bu_{\U^*+\tau \U})(\bu_{\U^*+\tau \U},\A\widetilde{\bu})\leq\tau\|\U\|_{\H}\|\A\widetilde{\bu}\|_{\H}\leq\frac{\nu}{8}\|\A\widetilde{\bu}\|_{\H}^2+\frac{2\tau^2}{\nu}\|\U\|_{\H}^2,
		\end{align}
		where we have used the fact that $\bu_{\U^*+\tau \U}(t)\in\mathcal{M}$ for all $t\in[0,T]$. Using this fact, and  applying H\"older's and Young's inequalities, we estimate $-(\U^*,\widetilde{\bu})(\bu_{\U^*+\tau \U},\A\widetilde{\bu})-(\U^*,\A\bu_{\U^*})\|\nabla\widetilde{\bu}\|_{\H}^2$  as
		\begin{align}\label{eqn-difference-10}
			-&(\U^*,\widetilde{\bu})(\bu_{\U^*+\tau \U},\A\widetilde{\bu})-(\U^*,\A\bu_{\U^*})\|\nabla\widetilde{\bu}\|_{\H}^2\nonumber\\&\leq\|\U^*\|_{\H}\|\widetilde{\bu}\|_{\H}\|\A\widetilde{\bu}\|_{\H}+\|\U^*\|_{\H}\|\A\bu_{\U^*}\|_{\H}\|\nabla\widetilde{\bu}\|_{\H}^2\nonumber\\&\leq\frac{\nu}{8}\|\A\widetilde{\bu}\|_{\H}^2+\frac{2}{\nu}\|\U^*\|_{\H}^2\|\widetilde{\bu}\|_{\H}^2+\left(\frac{1}{2}\|\U^*\|_{\H}^2+\frac{1}{2}\|\A\bu_{\U^*}\|_{\H}^2\right)\|\nabla\widetilde{\bu}\|_{\H}^2. 
		\end{align}
		Combining \eqref{eqn-difference-7}-\eqref{eqn-difference-10} and substituting it in \eqref{eqn-difference-6}, we deduce  for all $t\in[0,T]$ 
		\begin{align}
			&	\|\nabla\widetilde{\bu}(t)\|_{\H}^2+\nu\int_0^t\|\A\widetilde{\bu}(s)\|_{\H}^2\d s\nonumber\\&\leq \|\nabla\widetilde{\bu}(0)\|_{\H}^2+2\int_0^t\left[\nu\left(\|\nabla\bu_{\U^*+\tau \U}(s)\|_{\H}+\|\nabla\bu_{\U^*}(s)\|_{\H}\right)^2\right.\nonumber\\&\qquad\left.+\left(\frac{1}{2}\|\U^*(s)\|_{\H}^2+\frac{1}{2}\|\A\bu_{\U^*}(s)\|_{\H}^2\right)\right]\|\nabla\widetilde{\bu}(s)\|_{\H}^2 \d s\nonumber\\&\quad+2\int_0^t\left[\frac{C}{\nu}\|\nabla\bu_{\U^*}(s)\|_{\H}^4+\frac{2}{\nu}\|\U^*(s)\|_{\H}^2\right]\|\widetilde{\bu}(s)\|_{\H}^2\d s+\frac{8\tau^2}{\nu}\int_0^t\|\U(s)\|_{\H}^2\d s.
		\end{align}
		Since $\widetilde{\bu}(0)=\boldsymbol{0}$, applying  Gr\"onwall's inequality and using  \eqref{eqn-difference-9}, \eqref{eqn-difference-11} and \eqref{eqn-difference-8}, we deduce 
		\begin{align}
			&	\|\nabla\widetilde{\bu}(t)\|_{\H}^2+\nu\int_0^t\|\A\widetilde{\bu}(s)\|_{\H}^2\d s\nonumber\\&\leq\bigg\{2\sup_{t\in[0,T]}\|\widetilde{\bu}(t)\|_{\H}^2\int_0^T\left[\frac{C}{\nu}\|\nabla\bu_{\U^*}(t)\|_{\H}^4+\frac{2}{\nu}\|\U^*(t)\|_{\H}^2\right]\d t+\frac{8\tau^2}{\nu}\int_0^T\|\U(t)\|_{\H}^2\d t\bigg\} \nonumber\\&\quad\times\exp\bigg\{2\int_0^T\left[\nu\left(\|\nabla\bu_{\U^*+\tau \U}(t)\|_{\H}+\|\nabla\bu_{\U^*}(t)\|_{\H}\right)^2+\left(\frac{1}{2}\|\U^*(t)\|_{\H}^2+\frac{1}{2}\|\A\bu_{\U^*}(t)\|_{\H}^2\right)\right]\d t\bigg\}\nonumber\\&\leq \tau^2 C_{\nu,T}\left(\int_0^T\|\U(t)\|_{\H}^2\d t\right)\bigg[1+\left(\nu\|\nabla\bu_0\|_{\H}^2+\int_0^T\|\U^*(t)\|_{\H}^2\d t+\int_0^T\|\U(t)\|_{\H}^2\d t\right)^2\bigg]\nonumber\\&\quad\times
			\exp\bigg\{
			C_{\nu,T}	\bigg[1+\left(\nu\|\nabla\bu_0\|_{\H}^2+\int_0^T\|\U^*(t)\|_{\H}^2\d t+\int_0^T\|\U(t)\|_{\H}^2\d t\right)^2\bigg]\bigg\},
		\end{align}
		for all $t\in[0,T]$. 
	\end{proof}

The next lemma establishes the Fr\'echet differentiability of the mapping
$\U \mapsto \bu_{\U},$
which maps the admissible control set $\mathscr{U}_{\mathrm{ad}}$ into the state space $\mathrm{C}([0,T]; \H) \cap \mathrm{L}^2(0,T; \V)$.

	\begin{proposition}\label{lem3.8}
		Let $\bu_0\in\V\cap\mathcal{M}$  be given. Then the mapping $\U^* \mapsto \bu_{\U^*}$  from $\mathscr{U}_{\mathrm{ad}}$ into  $\mathrm{C}([0,T];\H)\cap\mathrm{L}^{2}(0,T;\V)$ is Fr\'echet differentiable. Moreover, we have
		\begin{equation}\label{lim}
			\left\{
			\begin{aligned}
				\lim_{\tau\downarrow 0}\frac{\|\bu_{\U^*+\tau \U} (t)- \bu_{\U^*}(t)-\tau\boldsymbol{\eta}(t)\|_{\H}}{\tau}&=0, \ \text{ for all }\ t\in[0,T],\\
				\lim_{\tau\downarrow 0}\frac{\|\bu_{\U^*+\tau \U} - \bu_{\U^*}-\tau\boldsymbol{\eta}\|_{\mathrm{L}^{2}(0,T;\V)}}{\tau}&=0,
			\end{aligned}
			\right.
		\end{equation}
		where $\boldsymbol{\eta}(\cdot)$ is the \emph{unique weak solution} of the linearized system:
		\begin{equation}\label{4a}
			\left\{
			\begin{aligned}
				\frac{\d\boldsymbol{\eta}(t)}{\d t}&=-\nu\A\boldsymbol{\eta}(t)+\nu\|\nabla\bu_{\U^*}(t)\|_{\H}^2\boldsymbol{\eta}(t)+2\nu(\nabla\bu_{\U^*}(t),\nabla\boldsymbol{\eta}(t))\bu_{\U^*}(t)\\&\quad-\B'(\bu_{\U^*}(t))\boldsymbol{\eta}(t)+\U(t)-(\U(t),\bu_{\U^*}(t))\bu_{\U^*}(t)\\&\quad-[(\U^*(t),\bu_{\U^*}(t))\boldsymbol{\eta}(t)+(\U^*(t),\boldsymbol{\eta}(t))\bu_{\U^*}(t)],\ \text{ in }\ \V',\\
				\boldsymbol{\eta}(0)&=\boldsymbol{0},
			\end{aligned}
			\right.
		\end{equation}
		for 	a.e. $t\in[0,T]$,  and $\bu_{\U^*}(\cdot)$  and $\bu_{\U^*+\tau\U}(\cdot)$ are the unique strong  solutions of the controlled system \eqref{eqn-constrained-NSE-N} with  the controls $\U^*$ and $\U^*+\tau\U$, respectively.

		Moreover, the following tangency property of linearized state  holds:  $$\boldsymbol{\eta}(t)\in T_{\bu_{\U^*}(t)}\mathcal{M}:=\left\{\w(t)\in\H:(\bu_{\U^*}(t),\w(t))=0\right\},$$ for all $t\in[0,T]$, that is,
	$
			(\boldsymbol{\eta}(t),\bu_{\U^*}(t))=0\ \text{ for all }\ t\in[0,T],
	$
		 so that $\boldsymbol{\eta}$ belongs to the tangent cone of feasible perturbations.
	\end{proposition}
	
	\begin{proof}
	To prove that the mapping
$	\U^* \mapsto \bu_{\U^*}$
	from $\mathscr{U}_{\mathrm{ad}}$ into $\mathrm{C}([0,T]; \H) \cap \mathrm{L}^2(0,T; \V)$ is Fr\'echet differentiable, it suffices to verify condition \eqref{lim}.
	Define
	$\y := \bu_{\U^* + \tau \U} - \bu_{\U^*} - \tau \boldsymbol{\eta}.$
	Then the function $\y(\cdot)$ satisfies the following system for a.e. $t\in[0,T]$:
		\begin{equation}\label{4b}
			\left\{
			\begin{aligned}
			\frac{\d \y(t)}{\d t}&=-\nu\A\y(t)+\nu\|\nabla\bu_{\U^*}(t)\|_{\H}^2\y(t)+\nu\left(\|\nabla\bu_{\U^*+\tau\U}(t)\|_{\H}^2-\|\nabla\bu_{\U^*}(t)\|_{\H}^2\right)\widetilde{\bu}(t)\\&\quad+\nu\left[\left(\|\nabla\bu_{\U^*+\tau\U}(t)\|_{\H}^2-\|\nabla\bu_{\U^*}(t)\|_{\H}^2\right)-2(\nabla\bu_{\U^*}(t),\nabla\widetilde{\bu}(t))\right]\bu_{\U^*}(t)\\&\quad+2\nu(\nabla\bu_{\U^*}(t),\nabla\y(t))\bu_{\U^*}(t)-\B'(\boldsymbol{u}_{\U^*}(t))\y(t)\\&\quad-\left[\B(\bu_{\U^*+\tau\U}(t))-\B(\bu_{\U^*}(t))-\B'(\bu_{\U^*}(t))\widetilde{\bu}(t)\right]\\&\quad-\tau(\U(t),\widetilde{\bu}(t))\bu_{\U^*+\tau\U}(t)-\tau(\U(t),\bu_{\U^*}(t))\widetilde{\bu}(t)-(\U^*(t),\bu_{\U^*}(t))\y(t)\\&\quad-(\U^*(t),\y(t))\bu_{\U^*}(t)-(\U^*(t),\widetilde{\bu}(t))\widetilde{\bu}(t), \\
				\y(0)&=\boldsymbol{0},
			\end{aligned}
			\right.
		\end{equation}
		in $\V'$, where $\widetilde{\bu}=\bu_{\U^*+\tau\U}-\bu_{\U^*}$.  Note that 
		\begin{align*}
			&\B(\bu_{\U^*+\tau\U})-\B(\bu_{\U^*})-\B'(\bu_{\U^*})(\bu_{\U^*+\tau\U}-\bu_{\U^*})\nonumber\\&=\B(\bu_{\U^*+\tau\U},\bu_{\U^*+\tau\U})-\B(\bu_{\U^*},\bu_{\U^*})\nonumber\\&\quad-\B(\bu_{\U^*+\tau\U}-\bu_{\U^*},\bu_{\U^*})-\B(\bu_{\U^*},\bu_{\U^*+\tau\U}-\bu_{\U^*})\nonumber\\&=\B(\bu_{\U^*+\tau\U}-\bu_{\U^*},\bu_{\U^*+\tau\U})-\B(\bu_{\U^*+\tau\U}-\bu_{\U^*},\bu_{\U^*})\nonumber\\&=\B(\bu_{\U^*+\tau\U}-\bu_{\U^*},\bu_{\U^*+\tau\U}-\bu_{\U^*})=\B(\widetilde{\bu},\widetilde{\bu}).
		\end{align*}
		Taking the inner product with $\y(\cdot)$ to the first equation in \eqref{4b}, we find for a.e. $t\in[0,T]$
		\begin{align}\label{3a27}
			&\frac{1}{2}\frac{\d}{\d t}\|\y(t)\|_{\H}^2+\nu\|\nabla\y(t)\|_{\H}^2\nonumber\\&=\nu\|\nabla\bu_{\U^*}(t)\|_{\H}^2\|\y(t)\|_{\H}^2+\nu\left(\|\nabla\bu_{\U^*+\tau\U}(t)\|_{\H}^2-\|\nabla\bu_{\U^*}(t)\|_{\H}^2\right)(\widetilde{\bu}(t),\y(t))
			\nonumber\\&\quad+\nu\left[\left(\|\nabla\bu_{\U^*+\tau\U}(t)\|_{\H}^2-\|\nabla\bu_{\U^*}(t)\|_{\H}^2\right)-2(\nabla\bu_{\U^*}(t),\nabla\widetilde{\bu}(t))\right](\bu_{\U^*}(t),\y(t))\nonumber\\&\quad+2\nu(\nabla\bu_{\U^*}(t),\nabla\y(t))(\bu_{\U^*}(t),\y(t))
			-\langle\B(\y(t),\bu_{\U^*}(t)),\y(t)\rangle\nonumber\\&\quad-\langle\B(\widetilde{\bu}(t),\widetilde{\bu}(t)),\y(t)\rangle -\tau(\U(t),\widetilde{\bu}(t))(\bu_{\U^*+\tau\U}(t),\y(t))\nonumber\\&\quad-\tau(\U(t),\bu_{\U^*}(t))(\widetilde{\bu}(t),\y(t))-(\U^*(t),\bu_{\U^*}(t))\|\y(t)\|_{\H}^2\nonumber\\&\quad-(\U^*(t),\y(t))(\bu_{\U^*}(t),\y(t))-(\U^*(t),\widetilde{\bu}(t))(\widetilde{\bu}(t),\y(t)),
		\end{align}
		where we have used the fact that $\langle\B(\bu_{\U^*},\y),\y\rangle=0$.	Using the Cauchy-Schwarz, triangle and Young  inequalities, we estimate the term $\nu\left(\|\nabla\bu_{\U^*+\tau\U}\|_{\H}^2-\|\nabla\bu_{\U^*}\|_{\H}^2\right)(\widetilde{\bu},\y)$ as
		\begin{align}\label{eqn-est-10}
			&	\nu\left(\|\nabla\bu_{\U^*+\tau\U}\|_{\H}^2-\|\nabla\bu_{\U^*}\|_{\H}^2\right)(\widetilde{\bu},\y)\nonumber\\&\leq\nu\left(\|\nabla\bu_{\U^*+\tau\U}\|_{\H}+\|\nabla\bu_{\U^*}\|_{\H}\right)\|\nabla(\bu_{\U^*+\tau\U}-\bu_{\U^*})\|_{\H}\|\widetilde{\bu}\|_{\H}\|\y\|_{\H}\nonumber\\&\leq \frac{\nu}{2}\|\widetilde{\bu}\|_{\H}^2\|\nabla\widetilde{\bu}\|_{\H}^2+\frac{\nu}{2}\left(\|\nabla\bu_{\U^*+\tau\U}\|_{\H}+\|\nabla\bu_{\U^*}\|_{\H}\right)^2\|\y\|_{\H}^2. 
		\end{align}
		Using the fact that 
		$\|\bu+\boldsymbol{h}\|_{\H}^2-\|\bu\|_{\H}^2-2(\bu,\boldsymbol{h})=\|\boldsymbol{h}\|_{\H}^2,$ for all $\bu,\boldsymbol{h}\in\H$, we estimate the term  $\nu\left[\left(\|\nabla\bu_{\U^*+\tau\U}\|_{\H}^2-\|\nabla\bu_{\U^*}\|_{\H}^2\right)-2(\nabla\bu_{\U^*},\nabla\widetilde{\bu})\right](\bu_{\U^*},\y)$ 
		by  applying  the Cauchy-Schwarz and Young inequalities as 
		\begin{align}\label{eqn-difference-20}
			&\nu\left[\left(\|\nabla\bu_{\U^*+\tau\U}\|_{\H}^2-\|\nabla\bu_{\U^*}\|_{\H}^2\right)-2(\nabla\bu_{\U^*},\nabla\widetilde{\bu})\right](\bu_{\U^*},\y)\nonumber\\&= \nu\|\nabla\widetilde{\bu}\|_{\H}^2(\bu_{\U^*},\y)\leq \nu\|\nabla\widetilde{\bu}\|_{\H}^2\|\y\|_{\H},
		\end{align}
		where we have used the fact that $\bu_{\U^*}(t)\in\mathcal{M}$ for all $t\in[0,T]$. Using the Cauchy-Schwarz and Young inequalities, we estimate the term $2\nu(\nabla\bu_{\U^*},\nabla\y)(\bu_{\U^*},\y)$  as 
		\begin{align}
			2\nu(\nabla\bu_{\U^*},\nabla\y)(\bu_{\U^*},\y)&\leq 2\nu\|\nabla\bu_{\U^*}\|_{\H}\|\nabla\y\|_{\H}\|\y\|_{\H}\leq\frac{\nu}{4}\|\nabla\y\|_{\H}^2+\frac{1}{\nu}\|\nabla\bu_{\U^*}\|_{\H}^2\|\y\|_{\H}^2. 
		\end{align}
		Using  H\"older's, Ladyzhenskaya's and Young's inequalities, we estimate $-\langle\B(\y,\bu_{\U^*}),\y\rangle$ as
		\begin{align}\label{3a28}
			-\langle\B(\y,\bu_{\U^*}),\y\rangle&\leq\|\nabla\bu_{\U^*}\|_{\H}\|\y\|_{\mathbb{L}^4}^2\leq\sqrt{2}\|\nabla\bu_{\U^*}\|_{\H}\|\y\|_{\H}\|\nabla\y\|_{\H}\nonumber\\&\leq\frac{\nu}{4}\|\nabla\y\|_{\H}^2+\frac{2}{\nu}\|\nabla\bu_{\U^*}\|_{\H}^2\|\y\|_{\H}^2.
		\end{align}
		A similar calculation yields
		\begin{align}
			-\langle\B(\widetilde{\bu},\widetilde{\bu}),\y\rangle\leq\|\nabla\y\|_{\H}\|\widetilde{\bu}\|_{\mathbb{L}^4}^2\leq\frac{\nu}{4}\|\y\|_{\V}^2+\frac{2}{\nu}\|\widetilde{\bu}\|_{\mathbb{L}^4}^4.
		\end{align}
		Using the Cauchy-Schwarz and Young inequalities, and applying  the fact that $\bu_{\U^*+\tau\U},\bu_{\U^*}\in\mathcal{M}$, we estimate 
		\begin{equation}\label{eqn-est-11}
			\begin{aligned}
				-\tau(\U,\widetilde{\bu})(\bu_{\U^*+\tau\U},\y)&\leq\tau\|\U\|_{\H}\|\widetilde{\bu}\|_{\H}\|\y\|_{\H}\leq\frac{\tau^2}{4}\|\widetilde{\bu}\|_{\H}^2+\|\U\|_{\H}^2\|\y\|_{\H}^2, \\
				-\tau(\U,\bu_{\U^*})(\widetilde{\bu},\y)&\leq\tau\|\U\|_{\H}\|\widetilde{\bu}\|_{\H}\|\y\|_{\H}\leq\frac{\tau^2}{4}\|\widetilde{\bu}\|_{\H}^2+\|\U\|_{\H}^2\|\y\|_{\H}^2, \\-(\U^*,\bu_{\U^*})\|\y\|_{\H}^2&\leq\|\U^*\|_{\H}\|\y\|_{\H}^2,\\
				-(\U^*,\y)(\bu_{\U^*},\y)&\leq\|\U^*\|_{\H}\|\y\|_{\H}^2,\\
				-(\U^*,\widetilde{\bu})(\widetilde{\bu},\y)&\leq\|\U^*\|_{\H}\|\widetilde{\bu}\|_{\H}^2\|\y\|_{\H}\leq\frac{1}{2}\|\widetilde{\bu}\|_{\H}^4+\frac{1}{2}\|\U^*\|_{\H}^2\|\y\|_{\H}^2.
			\end{aligned}
		\end{equation}
		
		Combining \eqref{eqn-est-10}-\eqref{eqn-est-11}, substituting it in \eqref{3a27}  and then integrating from $0$ to $t$, we obtain
		\begin{align}\label{3a31}
			&\|\y(t)\|_{\H}^2+\frac{\nu}{2}\int_0^t\|\nabla\y(s)\|_{\H}^2\d s\nonumber\\&\leq \nu\int_0^t\|\widetilde{\bu}(s)\|_{\H}^2\|\nabla\widetilde{\bu}(s)\|_{\H}^2\d s+\frac{4}{\nu}\int_0^t\|\widetilde{\bu}(s)\|_{\mathbb{L}^4}^4\d s+\tau^2\int_0^t\|\widetilde{\bu}(s)\|_{\H}^2\d s+\int_0^t\|\widetilde{\bu}(s)\|_{\H}^4\d s\nonumber\\&\quad+2\nu\int_0^t\|\nabla\widetilde{\bu}(s)\|_{\H}^2\|\y(s)\|_{\H}\d s +\int_0^t\Big[\nu\left(\|\nabla\bu_{\U^*+\tau\U}\|_{\H}(s)+\|\nabla\bu_{\U^*}(s)\|_{\H}\right)^2\nonumber\\&\qquad+\frac{6}{\nu}\|\nabla\bu_{\U^*}(s)\|_{\H}^2+\left(4\|\U(s)\|_{\H}^2+4\|\U^*(s)\|_{\H}+\|\U^*(s)\|_{\H}^2\right)\Big]\|\y(s)\|_{\H}^2\d s.
		\end{align}
		By defining $\eta(t)=\|\y(t)\|_{\H}^2+\frac{\nu}{2}\int_0^t\|\nabla\y(s)\|_{\V}^2\d s$ with $\alpha=\frac{1}{2}$, 	an application of the nonlinear generalization of Gr\"onwall inequality (\cite[Theorem 21]{SSD-2003}) in \eqref{3a31} yields
		\begin{align}\label{3a32}
			&\|\y(t)\|_{\H}^2+\frac{\nu}{2}\int_0^t\|\nabla\y(s)\|_{\V}^2\d s\nonumber\\&\leq \Bigg\{\bigg[\nu\int_0^T\|\widetilde{\bu}(t)\|_{\H}^2\|\nabla\widetilde{\bu}(t)\|_{\H}^2\d t+\frac{4}{\nu}\int_0^T\|\widetilde{\bu}(t)\|_{\mathbb{L}^4}^4\d t+\tau^2\int_0^T\|\widetilde{\bu}(t)\|_{\H}^2\d t+\int_0^T\|\widetilde{\bu}(t)\|_{\H}^4\d t\bigg]^{1/2}\nonumber\\&\quad+\nu\int_0^T\|\nabla\widetilde{\bu}(t)\|_{\H}^2\d t\Bigg\}^2\exp\left(\int_0^TA(t)\d t\right)
			\nonumber\\&\leq \Bigg\{2\bigg[\left(\nu+\frac{8}{\nu}\right)\sup_{t\in[0,T]}\|\widetilde{\bu}(t)\|_{\H}^2\int_0^T\|\nabla\widetilde{\bu}(t)\|_{\H}^2\d t+\tau^2T\sup_{t\in[0,T]}\|\widetilde{\bu}(t)\|_{\H}^2+T\sup_{t\in[0,T]}\|\widetilde{\bu}(t)\|_{\H}^4\bigg]\nonumber\\& \quad+2\nu^2\left(\int_0^T\|\nabla\widetilde{\bu}(t)\|_{\H}^2\d t\right)^2\Bigg\}\exp\left(\int_0^TA(t)\d t\right),
		\end{align}
		for all $t\in[0,T]$, where $$A=2\nu^2\|\nabla\bu_{\U^*+\tau\U}\|_{\H}^2+\left(2\nu^2+\frac{6}{\nu}\right)\|\nabla\bu_{\U^*}\|_{\H}^2+\left(4\|\U\|_{\H}^2+4\|\U^*\|_{\H}+\|\U^*\|_{\H}^2\right).$$ An application of \eqref{eqn-difference-9} and \eqref{eqn-difference-11} guarantee that $\int_0^TA(t)dt<\infty$. Passing $\tau\to 0$ in \eqref{3a32}  and using \eqref{eqn-difference-8} yield the required result \eqref{lim}.

		Let us now establish  that  $\boldsymbol{\eta}(t)\in T_{\bu_{\U^*}(t)}\mathcal{M}:=\left\{\w(t)\in\H:(\bu_{\U^*}(t),\w(t))=0\right\},$ for all $t\in[0,T],$ that is,
	$
			(\boldsymbol{\eta}(t),\bu_{\U^*}(t))=0\ \text{ for all }\ t\in[0,T]. 
	$ Since $\bu_{\U^*}\in \mathrm{C}([0,T];\V)$ and $\boldsymbol{\eta}\in\mathrm{C}([0,T];\H)$, we can choose these as test functions in the weak formulations of $\boldsymbol{\eta}$ and  $\bu_{\U^*},$ respectively. We define $\varphi(t)=(\boldsymbol{\eta}(t),\bu_{\U^*}(t))$ and consider  (\cite[Theoem 3.4]{Barbu-1993})
	\begin{align*}
		\frac{\d}{\d t}\varphi(t)&=\left\langle\partial_t\boldsymbol{\eta}(t),\bu_{\U^*}(t)\right\rangle+ \left(\boldsymbol{\eta}(t),\partial_t\bu_{\U^*}(t)\right)\nonumber\\&=\langle  -\nu\A\boldsymbol{\eta}(t)+\nu\|\nabla\bu_{\U^*}(t)\|_{\H}^2\boldsymbol{\eta}(t)+2\nu(\nabla\bu_{\U^*}(t),\nabla\boldsymbol{\eta}(t))\bu_{\U^*}(t)\\&\quad-\B'(\bu_{\U^*}(t))\boldsymbol{\eta}(t)+\U(t)-(\U(t),\bu_{\U^*}(t))\bu_{\U^*}(t)\\&\quad-[(\U^*(t),\bu_{\U^*}(t))\boldsymbol{\eta}(t)+(\U^*(t),\boldsymbol{\eta}(t))\bu_{\U^*}(t)], \bu_{\U^*}(t)\rangle\nonumber\\&\quad+(-\nu\A\bu_{\U^*}(t)+\nu\|\nabla\bu_{\U^*}(t)\|_{\H}^2\bu_{\U^*}(t)-\B(\bu_{\U^*}(t),\bu_{\U^*}(t))\nonumber\\&\quad+\U^*(t)-(\U^*(t),\bu_{\U^*}(t))\bu_{\U^*}(t), \boldsymbol{\eta}(t))
		\nonumber\\&=-2\nu(\nabla \boldsymbol{\eta}(t),\nabla \bu_{\U^*}(t))+2\nu \|\nabla\bu_{\U^*}(t)\|_{\H}^2 (\boldsymbol{\eta}(t), \bu_{\U^*}(t))\nonumber\\&\quad +2\nu(\nabla\bu_{\U^*}(t),\nabla\boldsymbol{\eta}(t))\|\bu_{\U^*}(t)\|_{\H}^2 \nonumber\\&\quad-\langle\B(\bu_{\U^*}(t),\boldsymbol{\eta}(t)),\bu_{\U^*}(t)\rangle-\langle\B(\bu_{\U^*}(t),\bu_{\U^*}(t)), \boldsymbol{\eta}(t) \rangle\nonumber\\&\quad+(\U(t),\bu_{\U^*}(t))-(\U(t),\bu_{\U^*}(t))\|\bu_{\U^*}(t)\|_{\H}^2
		\nonumber\\&\quad-2(\U^*(t),\bu_{\U^*}(t))(\bu_{\U^*}(t),\boldsymbol{\eta}(t))-(\U^*(t),\boldsymbol{\eta}(t))\|\bu_{\U^*}(t)\|_{\H}^2+(\U^*(t),\boldsymbol{\eta}(t))
		\nonumber\\&= 2\left[\nu \|\nabla\bu_{\U^*}(t)\|_{\H}^2 -(\U^*(t),\bu_{\U^*}(t))\right]\varphi(t),
	\end{align*}
	for a.e. $t\in[0,T]$, where we have used the fact that  $\bu_{\U^*}(t)\in\mathcal{M}$ for all $t\in[0,T]$. An application of the variation of constants formula yields 
	\begin{align}
		\varphi(t)=\varphi(0)\exp\left\{\int_0^t2\left[\nu \|\nabla\bu_{\U^*}(s)\|_{\H}^2 -(\U^*(s),\bu_{\U^*}(s))\right]\d s\right\},
	\end{align}
	for all $t\in[0,T]$. Since $\varphi(0)= (\boldsymbol{\eta}(0),\bu_{\U^*}(0))=0$ and 
	\begin{align*}
		\int_0^T\left[\nu \|\nabla\bu_{\U^*}(t)\|_{\H}^2 -(\U^*(t),\bu_{\U^*}(t))\right]\d t\leq \nu\int_0^T\|\nabla\bu_{\U^*}(t)\|_{\H}^2\d t+\int_0^T\|\U^*(t)\|_{\H}\d t<\infty, 
	\end{align*}
	we immediately have $\varphi(t)=(\boldsymbol{\eta}(t),\bu_{\U^*}(t))=0$ for all $t\in[0,T]$. 
	\end{proof}

	\begin{remark}
		The use of the nonlinear extension of Gr\"onwall's inequality allows us to avoid relying on the higher-order continuous dependence estimate given in \eqref{325a}. One can use Young's inequality to further estimate $\nu\left[\left(\|\nabla\bu_{\U^*+\tau\U}\|_{\H}^2-\|\nabla\bu_{\U^*}\|_{\H}^2\right)-2(\nabla\bu_{\U^*},\nabla\widetilde{\bu})\right](\bu_{\U^*},\y)$ given in \eqref{eqn-difference-11} as 
		\begin{align}
			\nu\left[\left(\|\nabla\bu_{\U^*+\tau\U}\|_{\H}^2-\|\nabla\bu_{\U^*}\|_{\H}^2\right)-2(\nabla\bu_{\U^*},\nabla\widetilde{\bu})\right](\bu_{\U^*},\y)\leq \frac{\nu}{2}\|\nabla\widetilde{\bu}\|_{\H}^4+\frac{\nu}{2}\|\y\|_{\H}^2. 
		\end{align}
		Substituting the above estimate into \eqref{3a31} and applying the standard Gr\"onwall inequality, we obtain the desired result by invoking the continuous dependence estimate given in \eqref{325a}.
	\end{remark}
	
We now proceed to prove our main result, Theorem \ref{main}, by invoking Lemmas \ref{lem3.7} and \ref{lem3.8}.
	\begin{proof}[Proof of Theorem \ref{main}]
		Let $(\bu_{\U^*}, \U^*) \in \mathscr{A}_{\mathrm{ad}}$ denote the optimal pair corresponding to the control problem \eqref{control problem}, as established in Theorem \ref{thm-exis-optimal}.
		Define the functional $\mathscr{G}(\U) = \mathscr{J}(\bu_{\U}, \U)$, where $(\bu_{\U}, \U)$ represents the unique weak solution of the controlled system \eqref{eqn-constrained-NSE-N} associated with the control $\U \in \mathscr{U}_{\mathrm{ad}}$.
		Assume that $\tau > 0$ is sufficiently small so that $\U^* + \tau \U \in \mathscr{U}_{\mathrm{ad}}$ and $(\bu_{\U^* + \tau \U}, \U^* + \tau \U) \in \mathscr{A}_{\mathrm{ad}}$.
		Under these conditions, the following relation holds (cf. \cite{MTM-2022}):
		\begin{align}\label{lam}
			&	\mathscr{G}(\U^*+ \tau \U) - \mathscr{G}(\U^*)\nonumber \\&= \mathscr{J}(\bu_{\U^* + \tau \U},{\U^* + \tau \U})- \mathscr{J}(\bu_{\U^*},{\U^*}) 
			\nonumber	\\ &= \frac{1}{2}\int_0^T  \|\bu_{\U^*+\tau \U} (t)- \bu_{\U^*}(t)\|^2_{\H} \d t+ \int_0^T (\bu_{\U^*+\tau \U} (t)- \bu_{\U^*}(t),\bu_{\U^*}(t) -\bu_d(t))\d t\nonumber\\&\quad+\frac{1}{2}\int_0^T\|\bu_{\U^*+\tau \U} (t)- \bu_{\U^*}(t)\|^2_{\V} \d t-\int_0^T\langle\bu_{\U^*+\tau \U} (t)+\bu_{\U^*}(t),\A\bu_{\U^*}(t)\rangle\d t\nonumber\\&\quad+ \int_0^T  \left[\tau(\U^*(t),\U(t))+\frac{\tau^2}{2}\|\U(t)\|_{\H}^2\right]\d t\nonumber\\&\quad +\frac{1}{2}\|\bu_{\U^*+\tau \U} (T)- \bu_{\U^*}(T)\|^2_{\H}+(\bu_{\U^* + \tau \U}(T)-\bu_{\U^* }(T),\bu_{\U^*}(T)-\bu^f) .
		\end{align}
		Using estimate \eqref{325a} (see Lemma \ref{lem3.7}), we know that both $\|\bu_{\U^*+\tau \U}(t) - \bu_{\U^*}(t)\|^2_{\H}$ for all $t \in [0,T]$ and $\|\bu_{\U^*+\tau \U} - \bu_{\U^*}\|^2_{\mathrm{L}^2(0,T;\V)}$ can be bounded by $C\tau^2 |\U|^2_{\mathrm{L}^2(0,T;\H)}$.
		Consequently, dividing by $\tau$ and letting $\tau \downarrow 0$, we deduce that
		$\|\bu_{\U^*+\tau \U} - \bu_{\U^*}\|_{\mathrm{L}^2(0,T;\H)} \to 0, \
		\|\bu_{\U^*+\tau \U} - \bu_{\U^*}\|_{\mathrm{L}^2(0,T;\V)} \to 0, \
		\text{and} \quad
		\|\bu_{\U^*+\tau \U}(T) - \bu_{\U^*}(T)\|_{\H} \to 0
		$
		as $\tau \downarrow 0$.
		We denote by $\mathscr{G}^{\prime}(\U^*, \U)$ the \emph{directional derivative} of $\mathscr{G}$ at $\U^*$ in the direction $\U$.	Let $\boldsymbol{\eta}(\cdot)$ be the solution to the linearized system \eqref{4a}.
		From Lemma \ref{lem3.8}, we also have the convergence stated in \eqref{lim}.
		Hence, dividing (\ref{lam}) by $\tau$, letting $\tau \downarrow 0$ and performing integration by parts (\cite[Theoem 3.4]{Barbu-1993}), we obtain the following result:
		\begin{align}\label{329}
			0 &\leq \mathscr{G}'(\U^*(t),\U(t)) \nonumber\\&= \underset{\tau \downarrow 0}{\lim} \frac{\mathscr{G}(\U^*(t)+ \tau \U(t))- \mathscr{G}(\U^*(t))}{\tau} \nonumber\\
			&= \int_0^T  \langle\boldsymbol{\eta}(t),\bu_{\U^*}(t) -\bu_d(t)+\A\bu_{\U^*}(t)\rangle\d t \nonumber\\&\quad+ \int_0^T(\U^*(t),\U(t)) \d t
			+(\boldsymbol{\eta}(T),\bu_{\U^*}(T)-\bu^f)	\nonumber\\&= \int_0^T \big \langle \boldsymbol{\eta}(t),-\partial\p_t(t)+\nu\A\p(t)-\nu\|\nabla\bu_{\U^*}(t)\|_{\H}^2\p(t)-2\nu(\bu_{\U^*}(t),\p(t))\A\bu_{\U^*}(t)\nonumber\\&\qquad+(\B'(\bu_{\U^*}(t)))^*\p (t)+\left[(\U^*(t),\bu_{\U^*}(t))\p(t)+\U^*(t)(\bu_{\U^*}(t),\p(t))\right]\big\rangle\d t\nonumber\\&\quad+ \int_0^T(\U^*(t),\U(t)) \d t+(\boldsymbol{\eta}(T),\bu_{\U^*}(T)-\bu^f)\nonumber\\&= \int_0^T\big\langle\partial_t\boldsymbol{\eta}(t)+\nu\A\boldsymbol{\eta}(t)-\nu\|\nabla\bu_{\U^*}(t)\|_{\H}^2\boldsymbol{\eta}(t)-2\nu(\nabla\bu_{\U^*}(t),\nabla\boldsymbol{\eta}(t))\bu_{\U^*}(t)\nonumber\\&\qquad+\B'(\bu_{\U^*}(t))\boldsymbol{\eta}(t)+[(\U^*(t),\bu_{\U^*}(t))\boldsymbol{\eta}(t)+(\U^*(t),\boldsymbol{\eta}(t))\bu_{\U^*}(t)],\p(t)\big\rangle\d t\nonumber\\&\quad+ \int_0^T(\U^*(t),\U(t)) \d t\nonumber\\&= \int_0^T(\U (t)-(\U(t),\bu_{\U^*}(t))\bu_{\U^*}(t),\p(t))\d t  + \int_0^T(\U^*(t),\U(t)) \d t.
		\end{align}
Therefore, from \eqref{329}, we infer 
		$$0 \leq \mathscr{G}'(\U^*(t), \U(t))= \int_0^T(\U (t)-(\U(t),\bu_{\U^*}(t))\bu_{\U^*}(t),\p(t))\d t  + \int_0^T(\U^*(t),\U(t)) \d t,$$  for all $\U\in\mathrm{L}^2(0,T;\H)$. Hence, from the above relation, it is immediate that
		\begin{equation}\label{3a37}
			\U^*(t)=-\p(t)+(\p(t),\bu_{\U^*}(t))\bu_{\U^*}(t), \ \text{ for a.e. } \ t\in[0,T]\ \text{ in }\ \H,
		\end{equation}
		which completes the proof. Since $\p\in \mathrm{C}([0,T];\H)$, the above relation holds true for all $t\in[0,T]$. 
	\end{proof}	
	
	\begin{remark}
		Since $(\boldsymbol{\eta},\bu_{\U^*})=0$ and  	$
			\U^*=-\p+(\p,\bu_{\U^*})\bu_{\U^*}, $ we find 
			\begin{align}
				(\U^*, \boldsymbol{\eta})=(-\p+(\p,\bu_{\U^*})\bu_{\U^*},\boldsymbol{\eta})=-(\p,\boldsymbol{\eta}). 
			\end{align}
	\end{remark}
	
	\begin{remark}
	Let us consider	the cost functional
		\begin{equation}\label{cost-1}
			\begin{aligned}
				\mathscr{J}(\bu,\U) &:= \frac{1}{2} \int_0^T\|\bu(t)-\bu_d\|_{\H}^2 \d t+\frac{1}{2}\int_0^T\|\nabla\bu(t)\|_{\H}^2\d t+ \int_0^Th(\U(t))\d t\\&\quad+\frac{1}{2}\|\bu(T)-\bu^f\|_{\H}^2,
			\end{aligned}
		\end{equation}
		subject to the constraint  \eqref{eqn-constrained-NSE-N}, under the following assumption on  $h(\cdot)$:
		\begin{hypothesis}\label{hyp}
			The function $h:\H\to(-\infty,\infty)$ is convex and  lower semicontinuous. Moreover, it satisfies the coercivity condition: \begin{align}\label{2a19}h(\U)\geq \kappa_1\|\U\|_{\H}^2+\kappa_2, \ \text{ for all } \ \U\in\mathbb{H}, \end{align} for some $\kappa_1>0$ and $\kappa_2\in\mathbb{R}$.
		\end{hypothesis}
		Then the optimality condition given in \eqref{3.41} becomes 
		\begin{align}
			\partial h(\U^*(t))\in-\p(t)+(\p(t),\bu^*(t))\bu^*(t)=-\pi_{\bu^*}\p(t), \ \text{ for a.e. } \ t\in[0,T]\ \text{ in }\ \H, 
		\end{align}
		where $\partial h(\cdot)$ denotes the subdifferential of $h(\cdot)$. Moreover, for  $\f(t)\in \partial h(\U^*(t))$, we have  $(\f(t),\bu^*(t))=0$ for a.e. $t\in[0,T],$ since $\bu^*(t)\in\mathcal{M}$ for all $t\geq 0$, that is, $\f(t)\in T_{\bu^*(t)}\mathcal{M}$ for a.e. $t\in[0,T].$
	\end{remark}

	\subsection{Conclusion}
	The results of Theorem \ref{main} admit a natural geometric interpretation in terms of the underlying manifold structure of the state space. The control system is constrained to evolve on the manifold $\mathcal{M}=\{\bu\in\H:\|\bu\|_{\H}=1\}\subset\H$, so that admissible perturbations of the state lie in the tangent space $T_{\bu}\mathcal{M}:=\{\w \in \H : (\w,\bu)=0\}$. The optimal control $\U^*(t)$ is obtained by projecting the adjoint variable $\p(t)$ onto this tangent space, yielding the relation $\U^*(t) = -\pi_{\bu^*(t)}(\p(t)) = -\p(t) + (\p(t),\bu^*(t))\bu^*(t)$, which ensures that the control is always tangent to the manifold and thus preserves the state constraint. The optimal control $\U^*(t)$ is then the projection of the adjoint vector $\p(t)$ onto this tangent space, it points in the direction along the sphere where the cost functional decreases most steeply, while never pushing the state off the sphere.  The adjoint variable $\p(t)$ itself encodes the gradient of the cost functional with respect to the state trajectory, and its backward evolution captures the sensitivity of the terminal and running costs along feasible directions. Meanwhile, the linearized variable $\boldsymbol{\eta}(t)$, which solves the linearized system around the optimal trajectory, represents infinitesimal perturbations of the state induced by variations in the control. Geometrically, $\boldsymbol{\eta}(t)$ remains tangent to the manifold $T_{\bu_{\U^*}(t)}\mathcal{M}$ for all $t$, reflecting the fact that only perturbations preserving the state constraint are admissible. Together, these structures form a coherent geometric picture, the optimal control aligns with the adjoint variable projected onto the tangent space, while the linearized state explores feasible directions within the tangent cone, ensuring consistency with the manifold constraint throughout the evolution.

	\medskip
	\noindent
	\textbf{Acknowledgments:}  Support for M. T. Mohan's research received from the National Board of Higher Mathematics (NBHM), Department of Atomic Energy, Government of India (Project No. 02011/13/2025/NBHM(R.P)/R\&D II/1137). 
	
%
%
%
%
%
%
%
	
\bibliographystyle{amsplain}
\bibliography{References}

\end{document}